\pdfoutput=1
\documentclass[]{article}
\usepackage{amssymb,hyperref,verbatim,cite,graphicx}
\usepackage{amsmath}
\usepackage{amsthm}
\newtheorem{proposition}{Proposition}
\newtheorem{theorem}{Theorem}
\newtheorem{corollary}{Corollary}
\usepackage{authblk}
\usepackage{float}
\newfloat{Chart}{htbp}{grf}
\makeatletter \let\cl@chapter\relax \makeatother
\usepackage{multirow, longtable}
\usepackage{booktabs}
\usepackage{multirow}
\usepackage{diagbox}
\usepackage{comment}
\usepackage{pdflscape}
\usepackage{mathtools}

\usepackage{longtable}
\usepackage{threeparttable}
\usepackage{enumerate}
\usepackage[dvipsnames]{xcolor}
\usepackage{tikz}
\tikzstyle{startstop} = [rectangle, rounded corners, minimum width=3cm, minimum height=1cm,text centered, draw=black]
\tikzstyle{io} = [rectangle, rounded corners, minimum width=3cm, minimum height=1cm,text centered, draw=black]
\tikzstyle{process} = [rectangle, minimum width=2cm, minimum height=1cm, text centered, draw=black]
\tikzstyle{thread} = [rectangle, minimum width=2cm, minimum height=1cm, text centered, draw=black]
\tikzstyle{parallel} = [rectangle, minimum width=3cm, minimum height=1cm, text centered, draw=black]
\tikzstyle{arrow} = [thick,->,>=stealth]
\usepackage{pgfplots}
\pgfplotsset{width=10cm,compat=1.9}
\usepackage[ruled,vlined,linesnumbered]{algorithm2e}
\usepackage{hyperref}
\usepackage{subcaption}
\hypersetup{
    colorlinks=true,
    linkcolor=blue,
    filecolor=magenta,      
    urlcolor=cyan,
}
\usepackage{cleveref, xurl}

\begin{document}

\newcommand{\red}{\color{red}}
\newcommand{\blue}{\color{blue}}
\newcommand{\green}{\color{green}}
\allowdisplaybreaks

\title{An Alternating Primal Heuristic for Nonconvex MIQCQP with Dynamic Convexification and Parallel Local Branching}
\author[1]{Yongzheng Dai}
\author[2]{Chen Chen}
\affil[1]{ISyE, Georgia Institute of Technology, Atlanta, GA, USA}
\affil[2]{ISE, The Ohio State University, Columbus, OH, USA}
\date{}
\maketitle

\begin{abstract}
We develop a novel primal heuristic for nonconvex Mixed-Integer Quadratically Constrained Quadratic Programs (MIQCQPs). The method is built around a convex approximation that is dynamically adjusted within a feasibility-pump-style alternating heuristic. Approximations are adjusted based on the structure of the MIQCQP instance.  Additionally, parallelized local branching is incorporated to further refine detected solutions. This paper builds upon the second-place finalist submission in the 2025 Land-Doig MIP Computational Competition. Our results are validated with computational experiments on instances from QPLIB, finding feasible solutions for three previously unsolved cases and improving the best-known solutions for fifteen instances within five minutes of runtime.
\end{abstract}

\section{Introduction}\label{sec:intro}

Consider a bounded Mixed-Integer Quadratically-Constrained Quadratic Programming (MIQCQP) problem,

\begin{subequations}\label{eq:miqcqp}
    \begin{align}
        \min\ &x^TQ^0x + a_0^Tx + b_0\label{eq:obj} \\
        \mbox{s.t.}\quad &x^TQ^kx + a_k^Tx \leq b_k,\ \forall k\in\{1,...,m_1\}\label{eq:quadcon}\\
        &Ax \leq b^{A},\label{eq:linearcon}\\
        &\ell \leq x \leq u,\label{eq:boxcon}\\
        &x_j\in \mathbb{Z},\ \forall j\in\mathcal{I}, \label{eq:intcon}
    \end{align}
\end{subequations}
The parameters are $Q^k \in \mathbb{R}^{n\times n}$, $a_k\in \mathbb{R}^n$ for all $k\in\{0,...,m_1\}$, $b_0\in \mathbb{R}$, $b\in \mathbb{R}^{m_1}$, $A\in \mathbb{R}^{m_2\times n}$, $b^A\in \mathbb{R}^{m_2}$, $\ell \in \mathbb{R}^n$, and $u \in \mathbb{R}^n$. Variables are given by $x\in \mathbb{R}^n$ with  integer entries $x_j\in \mathbb{Z}$ for $j\in \mathcal{I} \subseteq \mathcal{N} = \{1,...,n\}$. Furthermore, w.l.o.g., we assume $Q^k\in \mathbb{S}^n$ is a symmetric matrix, for all $k\in\{1,...,m_1\}$. Let $\mathcal{B} := \{j\in \mathcal{I}\ |\ \ell_j = 0 \wedge u_j= 1\}$ be the set of indices for all binary variables (integer variables with $l_j=0, u_j=1$). If $Q^k\succeq 0$ for all $k\in\{0,...,m_1\}$, Problem~(\ref{eq:miqcqp}) is a convex MIQCQP, whose continuous relaxation is convex; otherwise, we refer to it as a nonconvex MIQCQP.

We categorize MIQCQPs into three classes:

\begin{enumerate}
    \item \textbf{MIBQP}: mixed-integer box-constrained quadratic programs. MIQCQP with a quadratic objective function (\ref{eq:obj}), box constraints (\ref{eq:boxcon}), and integer constraints (\ref{eq:intcon}), and without any other linear or quadratic constraints.
    \item \textbf{MIQP}: mixed-integer quadratic programming. MIQCQP with quadratic objective function (\ref{eq:obj}) and at least one general linear constraint (\ref{eq:linearcon}), and without quadratic constraints. 
    \item \textbf{MIQCP}: mixed-integer quadratically constrained programming. MIQCQP with at least one quadratic constraint ($m_1 \geq 1$).
\end{enumerate}

These classes are substantially different from the perspective of a feasibility heuristic. Generating feasible solutions for MIBQPs is trivial (though NP-hard in general to approach optimality), as any point satisfying the variable bounds and integrality conditions is feasible. For MIQPs, a feasible solution can be detected using existing primal heuristics developed for Mixed-Integer Linear Programs (MILPs), such as those in \cite{fischetti2005feasibility,fischetti2003local,pisinger2018large}. However, for MIQCPs, the presence of nonconvex quadratic constraints makes feasible solution detection significantly more challenging and thus computationally expensive (see, e.g. \cite{wiese2021computational}). 

In this work, we develop a primal heuristic to find feasible solutions for nonconvex MIQCQPs, which consists of two stages: (i) detecting a feasible solution with an alternating method and (ii) refining the feasible solution with local branching. The alternating method involves a dynamic, eigenvalue-based convex approximation scheme that iteratively refines approximations of nonconvex quadratic terms. This scheme is integrated with integer rounding in a feasibility-pump-like alternating heuristic approach; subsequently, the feasible solution is refined with parallel local branching.

\subsection{Literature Review}
MIQCQP generalizes both Mixed-Integer Linear Programming (MILP) and (continuous) Quadratically Constrained Quadratic Programming (QCQP), enabling the representation of problems that combine discrete structures and nonlinear relationships. Such problems arise in a wide range of real-world applications, including power systems \cite{bai2015decomposition,liu2018global}, water systems \cite{faria2011novel}, pooling and mixing processes \cite{castro2016source,beach2020approximate}, signal processing \cite{pan2021joint}, finance \cite{cui2013convex}, and machine learning \cite{bertsimas2009algorithm}. 

There is a wide range of solution techniques for MIQCQPs, but this paper focuses primarily on primal heuristics (see \cite{berthold2018computational} for a recent survey), i.e. algorithms that aim (but are not guaranteed) to find high-quality feasible solutions quickly, e.g. \cite{chicco2020metaheuristic,lopez2013heuristic,cattaruzza2024exact}. The work in this paper was initiated by Yongzheng Dai as a submission for the \href{https://github.com/foreverdyz/primalheuristic\_miqcqp}{Land-Doig MIP Computational Competition 2025}\footnote{https://www.mixedinteger.org/2025/competition} on primal heuristics for MIQCQP. This competition included submissions such as a Frank-Wolfe-based primal heuristic \cite{mexi2025frank} and a classification-and-relaxation method. This paper was inspired especially by primal heuristics that have been successfully incorporated into MILP solvers: the feasibility pump \cite{fischetti2005feasibility}, local branching \cite{fischetti2003local,nannicini2008local}, large neighborhood search \cite{pisinger2018large}, and rounding heuristics \cite{nannicini2012rounding}. Many attempts have been made to extend these techniques nonlinear settings, and this paper is an additional contribution to this literature. Specifically, our approach involves a unique combination of dynamic convex approximation (e.g. \cite{gomez2025solving,castro2015tightening,beach2022compact,wiese2021computational,bienstock2025accurate}) , alternating heuristic (e.g.  \cite{d2012storm, belotti2017three}), and local branching (e.g. \cite{liberti2011recipe}) for nonconvex MIQCQP. 

We also leverage parallelization in local branching, and so contribute to the growing literature on incorporating parallelization in Mixed-Integer Programming (MIP) solvers (see, e.g., \cite{bixby2007progress,berthold2012solving,achterberg2013mixed,koch2022progress}), and parallel primal MIP heuristics in particular \cite{powley1990parallel, linderoth2001parallel, salvagnin2025fix, mexi2025frank}. We note also substantial efforts in massively parallel computing for MIP, e.g. \cite{phillips2006massively,eckstein2015pebbl,shinano2016solving,koch2012could,shinano2018fiberscip,perumalla2021design}, as well as parallel presolving \cite{gleixner2023papilo,dai2023parallelized,dai2024serial}.

\subsection{Organization}

The remainder of the paper is organized as follows. Section~\ref{sec:convex_approx} discusses existing convex reformulations or relaxations to nonconvex MIQCQPs and our proposed convex approximation. Section~\ref{sec:primal_heuristics} presents the novel primal heuristics based on the proposed convex approximation, and Section~\ref{sec:PLB} develops the parallel local branching. Finally, numerical experiments is shown in Section~\ref{sec:experiment}.

\section{Convex Approximations}\label{sec:convex_approx}

This section begins with established convex reformulations for nonconvex MIQCQPs and the sufficient conditions under which these reformulations are relaxations. Then we propose a convex approximation and a novel fixed-point approach to iteratively refine the approximation.

\subsection{Established Convex Reformulations}

Consider following reformulation of Problem~(\ref{eq:miqcqp}) from \cite{billionnet2016exact}:

\begin{subequations}\label{eq:convex_miqcqp}
    \begin{align}
        \min\ &x^T(Q^0 + P^0)x + a_0^Tx + b_0 - \langle X, P^0\rangle\label{eq:cMIQCP_obj}\\
        \mbox{s.t.}\quad &x^T(Q^k + P^k)x + a_k^Tx - \langle X, P^k\rangle \leq b_k,\ \forall k\in\{1,...,m_1\} \label{eq:cMIQCP_qcon}\\
        &Ax \leq b^A\\
        & X = x\circ x \label{eq:cMIQCP_bilinear}\\
        &\ell \leq x \leq u,\\
        &x_j\in \mathbb{Z},\ \forall j\in\mathcal{I},
    \end{align}
\end{subequations}
where $X = x\circ x$ is equivalent to $X_{ij} = x_ix_j$ for $\forall i,j\in\{1,...,n\}$. The idea is to perturb each indefinite quadratic coefficient matrix $Q^k\not\succeq 0$ with $P^k$ such that $Q^k+P^k \succeq 0$. This shifts the non-convexity in (\ref{eq:convex_miqcqp}) to the integrality of variables and bilinear constraints $X=x\circ x$. If either $x_i$ or $x_j$ is binary, $X_{ij} =x_ix_j$ can also be linearized by RLT \cite{sherali2002tight,bestuzheva2024efficient} inequalities  (see also McCormick \cite{mccormick1976computability}). For example, supposing $x_j$ is binary, we have
\begin{equation}\label{eq:brlt}
    \begin{aligned}
        \ell_ix_j\leq &X_{ij} \leq u_ix_j,\\
        x_i - u_i(1-x_j)\leq &X_{ij} \leq x_i - \ell_i(1-x_j).
    \end{aligned}
\end{equation}
This notion can be extended to bounded integer variables via binary (or unary) transformation.  For example,
\begin{equation}\label{eq:bexp}
    x_j = \sum_{h=0}^{\lfloor \log_2(u_j - \ell_j)\rfloor} 2^ht_{jh} + \ell_j,
\end{equation}
where $t$ are auxiliary binary variables. Thus $X_{ij} = x_ix_j$ can be rewritten in terms of more auxiliary variables of the form $z_{ijh} = t_{jh}x_i$, which can be in turn be linearized by Equation~(\ref{eq:brlt}).

Let $\mathcal{L}:=\{(i,j)\ |\ \exists k: P_{ij}^k\neq 0\}$, which denotes the set of all variable pairs that are perturbed with a non-zero coefficient in Problem~(\ref{eq:convex_miqcqp}); moreover, let $\mathcal{L}_c:=\{(i,j)\ |\ (\exists k: P_{ij}^k\neq 0) \wedge (i\not\in\mathcal{I}) \wedge (j\not\in\mathcal{I})\}$, which denotes the set of variable pairs that are perturbed in Problem~(\ref{eq:convex_miqcqp}) but cannot be linearized\footnote{Note that bilinear terms with unbounded integers also cannot be linearized, but we assume variable bounds throughout the paper. Hence $\mathcal{L}_c$ is only defined over continuous variables.} by Equation~(\ref{eq:brlt}).  If $\mathcal{L}_c = \emptyset$, Problem~(\ref{eq:convex_miqcqp}) can be linearized exactly by Equation~(\ref{eq:brlt}), which is a convex reformulation to the original MIQCQP~(\ref{eq:miqcqp}); this setting has been discussed in \cite{adams1986tight,billionnet2012extending,billionnet2016exact,wiese2021computational}. However, if $\mathcal{L}_c \neq \emptyset$, Equation~(\ref{eq:brlt}) is only a linear relaxation to $X_{ij} = x_ix_j$; thus, Problem~(\ref{eq:convex_miqcqp}) is a convex relaxation instead of reformulation to the original MIQCQP~(\ref{eq:miqcqp}). We will now consider how to construct this relaxation, and then adapt it to derive convex approximations.

\subsection{Eigenvalue-based Relaxation} \label{sec:org_eig_method}

Wiese \cite{wiese2021computational} summarized several approaches to select $P^k$ for reformulation~(\ref{eq:convex_miqcqp}) (when $\mathcal{L}_c = \emptyset$), including complete linearization, eigenvalue method, diagonal method, and maximizing the dual bound. For our primal heuristic (for the setting $\mathcal{L}_c \neq \emptyset$), we choose to focus on the eigenvalue method as it runs relatively fast in the reported results while providing effective relaxation bounds over a broad range of MIQCQP instances. 


The eigenvalue method \cite{hammer1970some} uses perturbing matrices $P^k := -\lambda^k I$ in the reformulation~(\ref{eq:convex_miqcqp}), where $I$ is the identity matrix, $\lambda^k := \min\{0, \lambda_{\min}^k\}$, and $\lambda_{\min}^k$ is the minimum eigenvalue of $Q^k$ for $k = 0,1,...,m_1$.

Now let us consider linearization of (\ref{eq:cMIQCP_bilinear}) in the eigenvalue method, extended to our setting. Since $P^k = -\lambda^k I$ only contains nonzero elements in the diagonal, all non-diagonal terms in $X$ can be ignored in the objective function (\ref{eq:cMIQCP_obj}) and quadratic constraints (\ref{eq:cMIQCP_qcon}). Therefore, we focus on the linearization for $X_{ii} = x_i^2$.

In our MIQCQP setting (\ref{eq:miqcqp}), all variables are bounded, so without loss of generality, we assume that all variables are nonnegative, i.e., $0\leq x\leq u$.

For all binary variables, i.e., $i\in\mathcal{B}$, we have $X_{ii} = x_i^2 = x_i$. Moreover, we can replace non-binary integer variables $x_i$ ($i\in\mathcal{I}\backslash\mathcal{B}$) with auxiliary binary variables $H_{h_1h_2}^i, t_{ih_1},t_{ih_2}$   (applying (\ref{eq:bexp}) and setting $\ell_i=0$): $x_i = \sum_{h=0}^{\lfloor \log_2(u_i)\rfloor}2^{h}t_{ih}, i\in \mathcal{I}\backslash\mathcal{B}$. Then
\[X_{ii} = x_i^2 = \sum_{h_1=0}^{\lfloor \log_2(u_i)\rfloor}\sum_{h_2=0}^{\lfloor \log_2(u_i)\rfloor}2^{h_1 + h_2}H_{h_1h_2}^i, H_{h_1h_2}^i := t_{ih_1}t_{ih_2}.\]
The binary bilinear term $H_{h_1h_2}^i = t_{ih_1}t_{ih_2}$ can be linearized (exactly) by RLT inequalities~(\ref{eq:brlt}), i.e.,
\begin{equation*}
    \begin{aligned}
        H_{h_1h_2}^i &= t_{ih_1}t_{ih_2},\\
        H_{h_1h_2}^i &\leq t_{ih_j}, j =1,2\\
        H_{h_1h_2}^i &\geq t_{ih_1} + t_{ih_2} - 1.
    \end{aligned}
\end{equation*}
As a result, $X_{ii}$ for $i\in\mathcal{I}\backslash \mathcal{B}$ can be linearized (exactly) as follows.
\begin{equation}
    \begin{aligned}
        \mathcal{H}(x_i):=\{X_{ii}\ |\ X_{ii} = &\sum_{h_1=0}^{\lfloor \log_2(u_i)\rfloor}\sum_{h_2=0}^{\lfloor \log_2(u_i)\rfloor}2^{h_1 + h_2}H_{h_1h_2}^{i},\\
        H_{h_1h_2}^i &\leq t_{ih_j}, j =1,2\\
        H_{h_1h_2}^i &\geq t_{ih_1} + t_{ih_2} - 1,\\
        x_i &= \sum_{h=0}^{\lfloor \log_2(u_i)\rfloor}2^{h}t_{ih},\\
        H_{h_1h_2}^i, &t_{ih_1},t_{ih_2}\in\{0,1\}
    \end{aligned}
\end{equation}
In our setting we must consider also continuous variables,  i.e., $i\notin\mathcal{I}$. Here we apply RLT (\ref{eq:brlt}) to $X_{ii} = x_{i}^2$ in order to obtain a relaxation. 

Now, plugging $P^k := -\lambda^k I$ into (\ref{eq:convex_miqcqp}) and linearize $X_{ii} = x_i^2$ for all $i=1,...,n$, we have the following relaxation:
\begin{subequations}\label{eq:eigen_I}
    \begin{align}
        \min\ &x^T(Q^0 - \lambda^0I)x + a_0^Tx + b_0 + \lambda^0\sum_{i=1}^n X_{ii} \label{eq:eigen_I_obj}\\
        \mbox{s.t.}\quad &x^T(Q^k -\lambda^k I)x + a_k^Tx + \lambda^k\sum_{i=1}^n X_{ii} \leq b_k,\ \forall k\in\{1,...,m_1\} \label{eq:eigen_I_qcon}\\
        &Ax \leq b^A,\label{eq:eigen_I_linear}\\
        &X_{ii} = x_i, \forall i\in\mathcal{B},\\
        &X_{ii} \in \mathcal{H}(x_i), \forall i\in\mathcal{I}\backslash \mathcal{B},\\
        &X_{ii} \leq u_ix_i, \forall i\notin\mathcal{I} \label{eq:rlt1}\\
        &X_{ii} \geq 2u_ix_i-u^2_i, \forall i\notin\mathcal{I} \label{eq:rlt2}\\
        &X_{ii} \geq 0, \forall i\notin\mathcal{I}\label{eq:rlt3}\\
        &0 \leq x \leq u,\\
        &x_j\in \mathbb{Z},\ \forall j\in\mathcal{I}\label{eq:eigen_I_end}.
    \end{align}
\end{subequations}

We can apply some variable reduction to this formulation. For any feasible solution $(x, X)$ in Problem~(\ref{eq:eigen_I}), we can fix $X_{ii}$ for $i\notin\mathcal{I}$ to its upper bound in (\ref{eq:rlt1}), i.e., $\bar{X}_{ii} := u_ix_i$, and keep $\bar{X}_{ii} := X_{ii}$ for $i\in\mathcal{I}$. Such fixing is valid in the following sense. The solution $(x, \bar{X})$ satisfies (\ref{eq:eigen_I_linear}-\ref{eq:eigen_I_end}). Given that $\lambda^k \leq 0$ for all $k\in\{0,1,...,m\}$, we can see that
\begin{equation}\label{eq:project_relation}
    \lambda^k\sum_{i\in\mathcal{I}} \bar{X}_{ii} + \lambda^k\sum_{i\notin\mathcal{I}} \bar{X}_{ii} \leq \lambda^k\sum_{i\in\mathcal{I}} X_{ii} + \lambda^k\sum_{i\notin\mathcal{I}} X_{ii}, \forall k=0,...,m.
\end{equation}
Equation~(\ref{eq:project_relation}) for $k=1,...,m$ shows $(x, \bar{X})$ satisfies (\ref{eq:eigen_I_qcon}) and hence $(x, \bar{X})$ is feasible in Problem~(\ref{eq:eigen_I}). Moreover, Equation~(\ref{eq:project_relation}) with $k=0$ implies the objective value of $(x, \bar{X})$ in (\ref{eq:eigen_I_obj}) is no greater than the objective value of $(x, X)$. Therefore, if Problem~(\ref{eq:eigen_I}) is feasible, there is at least one optimal solution $(x^*,X^*)$ such that $X_{ii}^* = u_ix_i^*$ for $i\notin\mathcal{I}$. 

Applying the fixings $X_{ii} := u_ix_i$ for $i\notin\mathcal{I}$ to Problem~(\ref{eq:eigen_I}) gives us the following reformulation:

\begin{subequations}\label{eq:eigen_miqcqp}
    \begin{align}
        \min\ &x^T(Q^0 - \lambda^0 I)x + a_0^Tx + b_0 + \lambda^0\sum_{i\in\mathcal{I}} X_{ii} + \lambda^0\sum_{i\notin\mathcal{I}} u_ix_i \label{eq:eigen_miqcqp_obj} \\
        \mbox{s.t.}\quad &x^T(Q^k - \lambda^k I)x + a_k^Tx + \lambda^k\sum_{i\in\mathcal{I}} X_{ii}\notag\\ 
        &+ \lambda^k\sum_{i\notin\mathcal{I}} u_ix_i \leq b_k,\ \forall k\in\{1,...,m_1\}\label{eq:eigen_miqcqp_qcon}\\
        &Ax \leq b^A\\
        &X_{ii} = x_i, \forall i\in\mathcal{B},\\
        &X_{ii} \in \mathcal{H}(x_i), \forall i\in\mathcal{I}\backslash \mathcal{B},\\
        &0 \leq x \leq u,\\
        &x_j\in \mathbb{Z},\ \forall j\in\mathcal{I}\label{eq:eigen_miqcqp_end}
    \end{align}
\end{subequations}

Problem~(\ref{eq:eigen_I}) is a relaxation to the original MIQCP~(\ref{eq:miqcqp}), and all feasible $x$ from (\ref{eq:eigen_I}) are also feasible in Problem~(\ref{eq:eigen_miqcqp}); thus Problem~(\ref{eq:eigen_miqcqp}) is also a relaxation to the original MIQCP~(\ref{eq:miqcqp}).

\subsection{Eigenvalue-based Approximation} \label{sec:eigen_method}

In this subsection we adapt the relaxation~(\ref{eq:eigen_miqcqp}) to develop principled mixed-integer convex approximations for MIQCPs. These approximations are used (via fixed-point method) to find feasible solutions for MIQCPs. The principle approach is to apply interval search over the variable bounds. Namely, we tighten the upper bounds, replacing $u_ix_i$ in (\ref{eq:eigen_miqcqp_obj}) and (\ref{eq:eigen_miqcqp_qcon}) with some $\hat{u}_ix_i$, where $0\leq \hat{u}\leq u$. This gives us the following $\hat{u}$-parameterized approximation:

\begin{subequations}\label{eq:approx_miqcqp}
    \begin{align}
        \mbox{Approx}(\hat{u}) = &\min_x \ x^T(Q^0 - \lambda^0 I)x + a_0^Tx + c + \lambda^0\sum_{i\in\mathcal{I}} X_{ii} + \lambda^0\sum_{i\notin\mathcal{I}} \hat{u}_ix_i \label{eq:approx_miqcqp_obj} \\
        \mbox{s.t.}\quad &x^T(Q^k - \lambda^k I)x + a_k^Tx + \lambda^k\sum_{i\in\mathcal{I}} X_{ii}\notag \\
        &+ \lambda^k\sum_{i\notin\mathcal{I}} \hat{u}_ix_i \leq b_k,\ \forall k\in\{1,...,m_1\} \label{eq:approx_miqcqp_qcon}\\
        &Ax \leq b^A\\
        &X_{ii} = x_i, \forall i\in\mathcal{B},\\
        &X_{ii} \in \mathcal{H}(x_i), \forall i\in\mathcal{I}\backslash \mathcal{B},\\
        &0 \leq x \leq u,\\
        &x_j\in \mathbb{Z},\ \forall j\in\mathcal{I},
    \end{align}
\end{subequations}


Let $(x(\hat{u}), X(\hat{u}))$ be a feasible solution to $\mbox{Approx}(\hat{u})$.  This approximation is constructed such that any feasible solution for which $\hat{u}_i \leq x_i(\hat{u}) \forall i\notin\mathcal{I}$ is also feasible to the original problem. This can be seen as follows. If $\hat{u}_i \leq x_i(\hat{u}) \forall i\not\in \mathcal{I}$, then together with the observations that $0\leq \hat{u}_ix_i(\hat{u}) \leq x_i(\hat{u})^2$ and $\lambda^k \leq 0$, we have
\begin{equation*}
    \begin{aligned}
        &x(\hat{u})^T(Q^k - \lambda^k I)x(\hat{u}) + a_k^Tx(\hat{u}) +\lambda^k\sum_{i\in\mathcal{I}} X_{ii}(\hat{u})+ \lambda^k\sum_{i\notin\mathcal{I}} \hat{u}_ix_i(\hat{u})\leq b_k,\\
        \implies &x(\hat{u})^T(Q^k - \lambda^k I)x(\hat{u}) + a_k^Tx(\hat{u}) + \lambda^k\sum_{i\in\mathcal{I}} X_{ii}(\hat{u})+ \lambda^k\sum_{i\notin\mathcal{I}} x_i(\hat{u})^2\leq b_k,\\
        \implies &x(\hat{u})^T(Q^k - \lambda^k I)x(\hat{u}) + a_k^Tx(\hat{u}) + \lambda^k\sum_{i=1}^n x_i(\hat{u})^2 \leq b_k,\\
        \implies &x(\hat{u})^T(Q^k)x(\hat{u}) + a_k^Tx(\hat{u})\leq b_k;
    \end{aligned}
\end{equation*}
and so $x(\hat{u})$ is a feasible solution to MIQCP~(\ref{eq:miqcqp}). 

Indeed, we can apply fixed-point iterations to find such a feasible solution (see e.g. \cite[Chapter 7]{ortega2000iterative}: Starting from an initial point $\hat{u}^0$, generate a sequence $\{\hat{u}^k\}$ by finding $x(\hat{u}^k)=\mbox{Approx}(\hat{u}^k)$ and setting the next iterate $\hat{u}_i^{k+1} := x_i(\hat{u}^k) \forall i\notin \mathcal{I}$. This continues until the consistency condition $\hat{u}_i^{k} \leq x_i(\hat{u}^k) \forall i\notin \mathcal{I}$ is attained. In Section~\ref{sec:MIQCQPproj}, we apply this fixed-point idea with acceleration (see \cite{walker2011anderson}): $\hat{u}_i^{k+1} = \alpha \hat{u}_i^k + (1-\alpha)x_i(\hat{u}^{k})$ for some $\alpha\in(0,1]$ for all $i\notin \mathcal{I}$.

\subsection{Approximation for MI(B)QPs}\label{sec:eigen_method_miqp}

If the problem is MIBQP or MIQP (no quadratic constraints) then the approximation design problem is underdetermined in the sense that for any $\hat{u}$ $\mbox{Approx}(\hat{u})$ shares the same feasible region (over $x$) as the original problem. In this case we focus more on objective value than feasibility. Namely, we modify the fixed-point method in Section~\ref{sec:eigen_method} to try to find some $\hat{u}$ such that (\ref{eq:approx_miqcqp_obj}) shares the same optimal solution to (\ref{eq:cMIQCP_obj})---the goal being that solving $\mbox{Approx}(\hat{u})$ can (hopefully) solve or at least provide a near-optimal solution to the original MIQP. We develop the following proposition to guide us on how to design such $\hat{u}$. 

\begin{proposition}\label{prop:fixed-point}
    Define $\mathrm{Q}_{\mathrm{MIQP}}:\min_x\ x^TQ^0x + a^Tx$ with $n_\alpha$ binary variables $x_\alpha$, $n_\beta$ continuous variables $x_\beta \geq 0$, and
    \[Q^0=\begin{bmatrix}
            Q_\alpha\ Q_\gamma \\
            Q_\gamma^T\ Q_\beta
        \end{bmatrix}, a = \begin{bmatrix}
            a_\alpha\\
            a_\beta
        \end{bmatrix}.\]
    Let $x(\hat{u})=(x_\alpha(\hat{u}), x_\beta(\hat{u}))$ be an optimal solution to $\mathrm{Approx}(\hat{u})$ for $\mathrm{Q}_{\mathrm{MIQP}}$. If $\hat{u}_\beta=2x_\beta(\hat{u})$, then $x_\beta(\hat{u})$ satisfies stationarity of $\mathrm{Q}_{\mathrm{MIQP}}$ for fixed binary variables $x_\alpha(\hat{u})$.
\end{proposition}
\begin{proof}
    Let
    \[f(x) := x^TQ^0x + a^Tx = x_\alpha^T Q_\alpha x_\alpha + 2x_\alpha^T Q_\gamma x_\beta + x_\beta^T Q_\beta x_\beta + a_\alpha^Tx_\alpha+a_\beta^Tx_\beta.\]
    Suppose $x^*=(x^*_\alpha,x^*_\beta)$ is an optimal solution to $\mathrm{Q}_{\mathrm{MIQP}}$. Fixing $x_\alpha = x_\alpha^*$, the first-order stationarity condition of $x_\beta^* = \arg\min\{f(x):x_\alpha = x_\alpha^*, x_\beta\geq 0\}$ is
    \begin{equation}\label{eq:opt_condition}
        \left\{\begin{aligned}
        \nabla_{(x_\beta)_i} f(x^*) = 0, &\mbox{ if } (x_\beta^*)_i> 0\\
        \nabla_{(x_\beta)_i} f(x^*) \geq 0, &\mbox{ if } (x_\beta^*)_i= 0
    \end{aligned} \right.
    \end{equation}
    where
    \begin{equation*}
       \nabla_{x_\beta} f(x^*) =  2Q_\beta x^*_\beta + 2Q_\gamma^T x_\alpha^* +a_\beta.
    \end{equation*}

    Now, let us write $\mbox{Approx}(\hat{u}):=\arg\min\{g(x, \hat{u})\}$, where \[g(x, \hat{u}):=x^T(Q^0-\lambda_{\min}^0 I)x+a^Tx+\lambda^0_{\min}\hat{u}_\beta^Tx_\beta + \lambda_{\min}^0\mathbf{1}^T x_\alpha.\] Fixing $x_{\alpha} = x_\alpha(\hat{u})$, $g(x,\hat{u})$ is convex from $Q^0-\lambda_{\min}^0 I \succeq 0$, and we have the first-order optimality condition
    \[\left\{\begin{aligned}
        \nabla_{(x_\beta)_i} g(x(\hat{u}), \hat{u}) = 0, &\mbox{ if } (x_\beta(\hat{u}))_i> 0\\
        \nabla_{(x_\beta)_i} g(x(\hat{u}), \hat{u}) \geq 0, &\mbox{ if } (x_\beta(\hat{u}))_i= 0
    \end{aligned} \right.,\]
    where
    \[\nabla_{x_\beta}g(x(\hat{u}), \hat{u}) = 2(Q_\beta-\lambda_{\min}^0 I_\beta)x_\beta(\hat{u}) + 2Q_\gamma^T x_\alpha(\hat{u}) + a_\beta +\lambda_{\min}^0\hat{u}_\beta.\]
    Plugging in the relation $\hat{u}_\beta=2x_\beta(\hat{u})$,
    \begin{equation*}
        \begin{aligned}
            \nabla_{x_\beta}g(x(\hat{u}), \hat{u}) &= 2(Q_\beta-\lambda_{\min}^0 I_\beta)x_\beta(\hat{u}) +2Q_\gamma^T x_\alpha(\hat{u}) + a_\beta +2\lambda_{\min}^0 x_\beta(\hat{u})\\
            &=2Q_\beta x_\beta(\hat{u})+2Q_\gamma^T x_\alpha(\hat{u}) + a_\beta\\
            &=\nabla_{x_\beta} f(x(\hat{u})).
        \end{aligned}
    \end{equation*}
    which implies $(x_\alpha(\hat{u}), x_\beta(\hat{u}))$ satisfies the stationarity condition~(\ref{eq:opt_condition}).
    \qed
\end{proof}

Note that we drop constraints in the Proposition for analytical tractability, but this idea for choosing $\hat u$ can be applied for more general linearly-constrained MIQP (moreover we can replace general integer variables with auxiliary binary variables as in Section~\ref{sec:org_eig_method}). 

\subsection{Redesigning Perturbations} \label{sec:matrix_select}

In this subsection, we consider the selection of $\lambda^k$ for the perturbing matrix $P^k:=-\lambda^k I$. As we mentioned before, the classic eigenvalue method \cite{hammer1970some} sets $\lambda^k :=\min\{0,\lambda_{\min}^k\}$, where $\lambda_{\min}^k$ is the minimum eigenvalue of $Q^k$. However, such a choice of $\lambda^k$ can be pathological in the context of our fixed-point-based approximations, as shown in Proposition~\ref{prop:org_lambda}.

\begin{proposition} \label{prop:org_lambda}
    Consider a nonconvex (symmetric) $\mathrm{MIQP}$, 
    \[\mathrm{Q}_{\mathrm{MIQP}}: \min_x\ x^TQ^0x + a^Tx\]
    with $n_\alpha$ binary variables $x_\alpha$, $n_\beta$ continuous variables $x_\beta\geq 0$, and
    \[Q^0=\begin{bmatrix}
            Q_\alpha\ Q_\gamma \\
            Q_\gamma^T\ Q_\beta
        \end{bmatrix}, a = \begin{bmatrix}
            a_\alpha\\
            a_\beta
        \end{bmatrix}.\]
    Let $\lambda_{\min}^0$ denote the minimum eigenvalue of $Q^0$, and denote the maximum and minimum eigenvalues of $Q_\beta$ as $\overline{\lambda}_\beta$ and $\underline{\lambda}_\beta$, respectively. 

Now, suppose there exists an optimal solution to the following perturbed problem \[\begin{split}x^*\in\arg\min\{x^T(Q^0-\lambda_{\min}^0I)x+a^Tx +\lambda_{\min}^0\hat{u}_\beta^Tx_\beta + \lambda_{\min}^0\mathbf{1}^T x_\alpha: \\x_\alpha\in\{0,1\}^{n_\alpha}, x_\beta\geq 0\}.\end{split}\] 
    
    Let $\beta^\prime :=\{i\in\beta: x_i^* > 0\}$. Then if, furthermore, $|\beta^\prime| = n_{\beta^\prime} > 0$, then there exists a corresponding continuous component $x^*_\beta$ that satisfies the following:
    \begin{enumerate}
        \item if $\lambda_{\min}^0 = \underline{\lambda}_\beta =\overline{\lambda}_\beta$, then $\|x^*_\beta\|$ can be arbitrarily large;
        \item if $\lambda_{\min}^0 = \underline{\lambda}_\beta <\overline{\lambda}_\beta$, then
        \[\frac{\lambda_{\min}^0}{2(\lambda_{\min}^0 - \overline{\lambda}_\beta)}\bigl| \|\hat{u}_{\beta^\prime}\| - \Omega \bigr| \leq \|x_\beta^*\| ;\]
        \item otherwise (i.e. $\lambda_{\min}^0 < \underline{\lambda}_\beta$), 
        \[\frac{\lambda_{\min}^0}{2(\lambda_{\min}^0 - \overline{\lambda}_\beta)}\bigl| \|\hat{u}_{\beta^\prime}\| - \Omega \bigr| \leq \|x_\beta^*\| \leq \frac{\lambda_{\min}^0}{2(\lambda_{\min}^0 - \underline{\lambda}_\beta)}(\|\hat{u}_{\beta^\prime}\| + \Omega),\]
    \end{enumerate}
    where $\Omega := \bigl\|(2Q_\gamma^T x_\alpha^* + a_{\beta})_{\beta^\prime}/\lambda_{\min}^0\bigr\|$; otherwise, $x_\beta^* = 0$.
\end{proposition} 
\begin{proof}
    Let $Q:= Q^0-\lambda_{\min}^0I = \begin{bmatrix}
            Q_\alpha-\lambda_{\min}^0I_\alpha, &Q_\gamma \\
            Q_\gamma^T, &Q_\beta-\lambda_{\min}^0I_\beta
        \end{bmatrix}$. We have 
        \[x^TQx = x_\alpha^T (Q_\alpha -\lambda_{\min}^0I_\alpha) x_\alpha + 2x_\alpha^T Q_\gamma x_\beta + x_\beta^T (Q_\beta-\lambda_{\min}^0I_\beta)x_\beta.\]
        Furthermore, we can split $a^Tx = a^T_\alpha x_\alpha + a^T_\beta x_\beta$ and get
        \[a^Tx+\lambda_{\min}^0\hat{u}_\beta^Tx_\beta + \lambda_{\min}^0\mathbf{1}^T x_\alpha = (a_\alpha + \lambda_{\min}^0\mathbf{1})^Tx_\alpha + (a_\beta + \lambda_{\min}^0\hat{u}_\beta)^Tx_\beta.\]
        Let $x_\alpha^* \in\{0,1\}^{n_\alpha}$ denote some optimal assignment of the binary variables, and $f(x_\beta) := x_\beta^T (Q_\beta-\lambda_{\min}^0I_\beta)x_\beta +2{x_\alpha^*}^T Q_\gamma x_\beta + (a_\beta+\lambda_{\min}^0\hat{u}_\beta)^Tx_\beta$. Then, 
        \begin{equation}\label{eq:xbeta_val}
            x_\beta^* = \arg\min\{f(x_\beta): x_\beta\geq 0\},
        \end{equation} 
        where $f$ is convex. Any optimal solution values of (\ref{eq:xbeta_val})  should satisfy the first-order optimality condition:
        \begin{equation}\label{eq:opt_condition_2}
            \left\{ \begin{aligned}
            \nabla_{(x_\beta)_i}f(x^*_\beta) = 0, &\mbox{ if } (x^*_\beta)_i > 0\\
            \nabla_{(x_\beta)_i}f(x^*_\beta) \geq 0, &\mbox{ if } (x^*_\beta)_i = 0
            \end{aligned}\right. .
        \end{equation}
        
        ~\\
        \textbf{Case 1.} Suppose $\lambda_{\min}^0 = \underline{\lambda}_\beta =\overline{\lambda}_\beta$. Here, $Q_\beta$ is symmetric with all identical eigenvalues, and so $Q_\beta = \lambda_{\min}^0I$. Plugging $Q_\beta-\lambda_{\min}^0I_\beta = 0$ into (\ref{eq:xbeta_val}), we have $x_\beta^* = \arg\min\{2{x_\alpha^*}^T Q_\gamma x_\beta + (a_\beta + \lambda_{\min}^0\hat{u}_\beta)^Tx_\beta : x_\beta \geq 0\}$, and then the problem reduces to a linear program. If the linear coefficients are all positive, the optimal solution $x_\beta^*$ is 0, which implies $n_{\beta^\prime} = 0$; otherwise, if some of the coefficients are negative, the problem is unbounded below; or if the coefficients are all zero, there are infinite optimal solutions, which can be arbitrarily large.
        
        ~\\
        \textbf{Case 2.} Suppose $\lambda_{\min}^0 = \underline{\lambda}_\beta <\overline{\lambda}_\beta$. Observe that $\lambda_{\min}^0 \leq \underline{\lambda}_\beta$ implies $Q_\beta- \lambda_{\min}^0I_\beta \succeq 0$, and the minimum and maximum eigenvalues of $Q_\beta-\lambda_{\min}^0I_\beta $ are $\underline{\lambda}_\beta -\lambda_{\min}^0$ and $\overline{\lambda}_\beta -\lambda_{\min}^0$, respectively. Let $(Q_\beta-\lambda_{\min}^0I_{\beta})_{\beta^\prime\beta^\prime}$ be the principal submatrix for indexes $\beta^\prime$. From (\ref{eq:opt_condition_2}), we have $\nabla_{x_{\beta^\prime}}f(x^*_{\beta^\prime}) = 0$, or
        \[2(Q_{\beta}- \lambda_{\min}^0I_{\beta})_{\beta^\prime\beta^\prime}x_{\beta^\prime}^* + (2Q_\gamma^T x_\alpha^* + a_\beta)_{\beta^\prime}+\lambda_{\min}^0\hat{u}_{\beta^\prime} = 0,\]
        \begin{equation}\label{eq:xbeta}
            \implies (Q_{\beta}- \lambda_{\min}^0I_{\beta})_{\beta^\prime\beta^\prime}x_{\beta^\prime}^* =  -\frac{1}{2}\lambda_{\min}^0\hat{u}_{\beta^\prime}-(Q_\gamma^T x_\alpha^* + \frac{1}{2}a_{\beta})_{\beta^\prime}.
        \end{equation}
        Since $(Q_\beta-\lambda_{\min}^0I_\beta)_{\beta^\prime\beta^\prime}$ is symmetric, we can apply a spectral-norm bound on the matrix:
        \begin{equation*}
            (\underline{\lambda}_{\beta^\prime\beta^\prime})\|x_\beta^*\|\leq \|(Q_\beta-\lambda_{\min}^0I_\beta)_{\beta^\prime\beta^\prime}x_{\beta^\prime}^*\|\leq (\overline{\lambda}_{\beta^\prime\beta^\prime})\|x_\beta^*\|,
        \end{equation*}
      
        \begin{equation}\label{eq:rangexbeta}
            \implies (\underline{\lambda}_\beta -\lambda_{\min}^0)\|x_\beta^*\|\leq \|(Q_\beta-\lambda_{\min}^0I_\beta)_{\beta^\prime\beta^\prime}x_{\beta^\prime}^*\|\leq (\overline{\lambda}_\beta -\lambda_{\min}^0)\|x_\beta^*\|,
        \end{equation}
        where the implication follows from interlaced eigenvalues with $Q$.
        
        Plugging (\ref{eq:xbeta}) into the right-hand inequality of (\ref{eq:rangexbeta}) together with $\|x_\beta^*\| = \|x_{\beta^\prime}^*\|$,
        
        \[\|-\frac{1}{2}\lambda_{\min}^0\hat{u}_{\beta^\prime}-(Q_\gamma^T x_\alpha^* + \frac{1}{2}a_{\beta})_{\beta^\prime}\| \leq (\overline{\lambda}_\beta -\lambda_{\min}^0)\|x_{\beta^\prime}^*\| = (\overline{\lambda}_\beta -\lambda_{\min}^0)\|x_\beta^*\|,\]
        \begin{equation}\label{eq:lowerxbeta}
            \begin{aligned}
                \implies \|x_\beta^*\|&\geq \frac{\|-\lambda_{\min}^0\hat{u}_{\beta^\prime}-(2Q_\gamma^T x_\alpha^* + a_{\beta})_{\beta^\prime}\|}{2(\overline{\lambda}_\beta - \lambda_{\min}^0)}\\
                &=\frac{-\lambda_{\min}^0\|\hat{u}_{\beta^\prime} + (2Q_\gamma^T x_\alpha^* + a_{\beta})_{\beta^\prime}/\lambda_{\min}^0\|}{2(\overline{\lambda}_\beta - \lambda_{\min}^0)}
            \end{aligned}
        \end{equation}

        Applying the reverse triangle inequality to the norm in Equation~(\ref{eq:lowerxbeta}) yields 
        \begin{equation}\label{eq:ubeta_range}
            \begin{aligned}
                \bigl| \|\hat{u}_{\beta^\prime}\| 
                 - \bigl\|(2Q_\gamma^T x_\alpha^* + a_{\beta})_{\beta^\prime}/\lambda_{\min}^0\bigr\| \bigr|
                &\leq \bigl\|\hat{u}_{\beta^\prime} + (2Q_\gamma^T x_\alpha^* + a_{\beta})_{\beta^\prime}/\lambda_{\min}^0\bigr\|,
            \end{aligned}
        \end{equation}
        
        
        and so
        \begin{equation}\label{eq:claim2_res}
            \frac{\lambda_{\min}^0}{2(\lambda_{\min}^0 - \overline{\lambda}_\beta)}\bigl| \|\hat{u}_{\beta^\prime}\| 
                 - \Omega \bigr| \leq \|x_\beta^*\|.
        \end{equation}

        \textbf{Case 3.} Here, $\lambda_{\min}^0 < \underline{\lambda}_\beta$. The analysis for Case 2 still holds true for Case 3; thus, Equation~(\ref{eq:claim2_res}) is true. 
        
        Plugging Equation~(\ref{eq:xbeta}) into the left-hand inequality of (\ref{eq:rangexbeta}) together with $\|x_\beta^*\| = \|x_{\beta^\prime}^*\|$, we have
        \begin{equation}\label{eq:upperxbeta}
        \begin{aligned}
            \|x_\beta^*\|&\leq \frac{\|-\lambda_{\min}^0\hat{u}_{\beta^\prime}-(2Q_\gamma^T x_\alpha^* + a_{\beta})_{\beta^\prime}\|}{2(\underline{\lambda}_\beta -\lambda_{\min}^0)}\\
            &= \frac{-\lambda_{\min}^0\|\hat{u}_{\beta^\prime} +(2Q_\gamma^T x_\alpha^* + a_{\beta})_{\beta^\prime}/\lambda_{\min}^0\|}{2(\underline{\lambda}_\beta -\lambda_{\min}^0)}\\
            &= \frac{\lambda_{\min}^0 \|\hat{u}_{\beta^\prime} +(2Q_\gamma^T x_\alpha^* + a_{\beta})_{\beta^\prime}/\lambda_{\min}^0\|}{2(\lambda_{\min}^0 - \underline{\lambda}_\beta)} 
        \end{aligned}
        \end{equation}
        Applying the triangle inequality to the norm in Equation~(\ref{eq:upperxbeta}) gives us
        \begin{equation}\label{eq:claim3_res}
            \|x_\beta^*\|\leq \frac{\lambda_{\min}^0}{2(\lambda_{\min}^0 - \underline{\lambda}_\beta)}(\|\hat{u}_{\beta^\prime}\| + \Omega).
        \end{equation}\qed
\end{proof}

From Proposition~\ref{prop:org_lambda}, for $\mbox{Approx}(\hat{u})$ of MIQP, $\|x(\hat{u})\|$ can be far larger than $\|\hat{u}\|$, especially when $\lambda_{\min}^0 \approx \underline{\lambda}_\beta$ or $\lambda_{\min}^0 \approx \underline{\lambda}_\beta\approx \overline{\lambda}_\beta$. Therefore, the fixed-point iteration $\hat{u}_{i}^{k+1} = 2x_i(\hat{u}^k)$ for MIQP may not converge and $\hat{u}_i \rightarrow u_i$ for some $i\notin\mathcal{I}$.  Thus we provide another option for perturbation with Theorem~\ref{thm:good_lambda}.

\begin{theorem} \label{thm:good_lambda}
    Consider the MIQP problem defined in Proposition~\ref{prop:org_lambda}. Now, suppose there exists an optimal solution to the following perturbed problem \[\begin{split}x^*\in\arg\min\{x^T(Q^0-\lambda_{s}^0I)x+a^Tx +\lambda_{s}^0\hat{u}_\beta^Tx_\beta + \lambda_{s}^0\mathbf{1}^T x_\alpha: \\x_\alpha\in\{0,1\}^{n_\alpha}, x_\beta\geq 0\},\end{split}\] 
    where $\lambda_s:=\min\{2\underline{\lambda}_\beta, \lambda_{\min}^0\} - 1$.
    Moreover, define $\overline{\Theta}:=\max\{\|2Q_\gamma^T x_\alpha + a_\beta\|\ |\ x_\alpha \in \{0,1\}^{n_\alpha} \}$; $\beta^\prime:=\{i\in\beta: x_i^* > 0\}$; and $\Omega:= \bigl\|(2Q_\gamma^T x_\alpha^* + a_\beta)_{\beta^\prime}/\lambda_s\bigr\|$.
    
    Then $x^*$ satisfies the following bound
    \[\|x_\beta^*-\hat{u}_\beta\| \leq \frac{(2\overline{\lambda}_\beta - \lambda_s)\|\hat{u}_\beta\| + \overline{\Theta}}{2(\underline{\lambda}_\beta-\lambda_s)}\]
    
    Moreover, if $|\beta^\prime| > 0$, then \[\frac{\lambda_s}{2(\lambda_s - \overline{\lambda}_\beta)}|\|\hat{u}_{\beta^\prime}\| - \Omega| \leq \|x_\beta^*\| \leq \frac{\lambda_s}{2(\lambda_s - \underline{\lambda}_\beta)}(\|\hat{u}_{\beta^\prime}\| + \Omega).\] 
\end{theorem} 

\begin{proof}
    From $\lambda_s<\lambda_{\min}^0 \leq \underline{\lambda}_\beta$, we have $Q_\beta-\lambda_sI_\beta \succ 0$; thus, the first-order optimality condition given by Equation~(\ref{eq:opt_condition_2}) holds. Consider $\hat{x}_\beta$ such that $\nabla_{\beta} f(\hat{x}_\beta) = 0$:
    \[2(Q_\beta- \lambda_s I_\beta)\hat{x}_\beta + 2Q_\gamma^T x_\alpha^* + a_\beta+\lambda_s\hat{u}_\beta = 0.\]
    \begin{equation}\label{eq:xbeta_2}
        \implies (Q_\beta- \lambda_sI_\beta)\hat{x}_\beta =  -\frac{1}{2}\lambda_s\hat{u}_\beta-Q_\gamma^T x_\alpha^* - \frac{1}{2}a_\beta.
    \end{equation}
    From first-order conditions~(\ref{eq:opt_condition_2}), we have
    \begin{equation}\label{eq:hat_opt}
            (x^*_\beta)_i = \left \{ \begin{aligned}&(\hat{x}_\beta)_i, &\mbox{ if } (\hat{x}_\beta)_i > 0\\
            &0, &\mbox{ if } (\hat{x}_\beta)_i \leq 0
            \end{aligned}\right. .
    \end{equation}
    Therefore, $\|x^*_\beta\| \leq \|\hat{x}_\beta\|$. Now, adding $-(Q_\beta -\lambda_sI_\beta)\hat{u}_\beta$ on both sides of (\ref{eq:xbeta_2}),
    \[(Q_\beta-\lambda_sI_\beta)(\hat{x}_\beta-\hat{u}_\beta) = (-Q_\beta+\frac{\lambda_s}{2}I_\beta)\hat{u}_\beta-Q_\gamma^T x_\alpha^* - \frac{1}{2}a_\beta;\]
    \[\implies \|(Q_\beta-\lambda_sI_\beta)(\hat{x}_\beta-\hat{u}_\beta)\| = \|(Q_\beta-\frac{\lambda_s}{2}I_\beta)\hat{u}_\beta-Q_\gamma^T x_\alpha^* - \frac{1}{2}a_\beta\|.\]
    Since $Q_\beta-\lambda_sI_\beta $ is symmetric positive definite, we can apply a spectral-norm bound on the matrix:
    \[\|(Q_\beta-\lambda_sI_\beta)(\hat{x}_\beta-\hat{u}_\beta)\|\geq (\underline{\lambda}_\beta-\lambda_s)\|\hat{x}_\beta-\hat{u}_\beta\|\]
    \begin{equation}\label{eq:thm1_condition}
        \begin{aligned}
            \implies \|\hat{x}_\beta-\hat{u}_\beta\| &\leq \frac{\|(Q_\beta-\frac{\lambda_s}{2}I_\beta)\hat{u}_\beta-Q_\gamma^T x_\alpha^* - \frac{1}{2}a_\beta\|}{\underline{\lambda}_\beta-\lambda_s}\\
            &\leq \frac{\|(Q_\beta-\frac{\lambda_s}{2}I_\beta)\hat{u}_\beta\| + \|Q_\gamma^T x_\alpha^* + \frac{1}{2}a_\beta\|}{\underline{\lambda}_\beta-\lambda_s}
        \end{aligned}
    \end{equation} 
    Now, as $\lambda_s < \underline{\lambda}_\beta$ then $Q_\beta-\frac{\lambda_s}{2}I_\beta \succ 0$; thus we have the spectral norm bound $\|(Q_\beta-\frac{\lambda_s}{2}I_\beta)\hat{u}_\beta\| \leq (\overline{\lambda}_\beta - \frac{\lambda_s}{2})\|\hat{u}_\beta\|$. Consequently,
    \begin{equation}\label{eq:upper_bound}
    \begin{aligned}
        \|\hat{x}_\beta-\hat{u}_\beta\| & \leq \frac{(\overline{\lambda}_\beta - \frac{\lambda_s}{2})\|\hat{u}_\beta\| + \|Q_\gamma^T x_\alpha^* + \frac{1}{2}a_\beta\|}{\underline{\lambda}_\beta-\lambda_s},\\
        &=  \frac{(2\overline{\lambda}_\beta - \lambda_s)\|\hat{u}_\beta\| + \overline{\Theta}}{2(\underline{\lambda}_\beta-\lambda_s)}.
    \end{aligned}
    \end{equation}

    Since $\hat{u}_\beta \geq 0$ (as defined at the beginning of Sec.~\ref{sec:eigen_method}) and given (\ref{eq:hat_opt}), we have $|(x^*_\beta)_i - (\hat{u}_\beta)_i| \leq |(\hat{x}_\beta)_i - (\hat{u}_\beta)_i|$; thus $\|x^*_\beta - \hat{u}_\beta\| \leq \|\hat{x}^*_\beta-\hat{u}_\beta\| \leq \frac{(2\overline{\lambda}_\beta - \lambda_s)\|\hat{u}_\beta\| + \overline{\Theta}}{2(\underline{\lambda}_\beta-\lambda_s)}$.

    Furthermore, due to $\lambda_s <\underline{\lambda}_\beta$, the argument of Case 3 in Proposition~\ref{prop:org_lambda} with $\lambda^0_{\min}$ replaced by $\lambda_s$; thus $\frac{\lambda_s}{2(\lambda_s - \overline{\lambda}_\beta)}|\|\hat{u}_{\beta^\prime}\| - \Omega| \leq \|x_\beta^*\| \leq \frac{\lambda_s}{2(\lambda_s - \underline{\lambda}_\beta)}(\|\hat{u}_{\beta^\prime}\| + \Omega)$.
    \qed
\end{proof}

The design of $\lambda_s := \min\{2\underline{\lambda}_\beta,\lambda_{\min}^0\}-1$ is explained as follows. First, to make $Q_0-\lambda_s I \succeq 0$, we need $\lambda_s \leq \lambda_{\min}\leq\min\{\underline{\lambda}_{\beta}, \underline{\lambda}_{\alpha}\}$. Second, we want $\frac{\lambda_s}{2(\lambda_s-\underline{\lambda}_{\beta})}\leq 1$ to make $\|x^{\beta}\|\leq \|\hat{u}_\beta +\Omega\|$ for convergence purposes. Therefore, considering $\lambda_s-\underline{\lambda}_{\beta}< 0$, 
\[\frac{\lambda_s}{2(\lambda_s-\underline{\lambda}_{\beta})} \leq 1\implies \lambda_s \geq 2(\lambda_s-\underline{\lambda}_{\beta})\implies \lambda_s\leq 2\underline{\lambda}_{\beta}.\]
Considering $\lambda_s \leq \lambda_{\min}\leq\min\{\underline{\lambda}_{\beta}, \underline{\lambda}_{\alpha}\}$, we have $\lambda_s \leq \min\{2\underline{\lambda}_\beta,\underline{\lambda}_\alpha\}$. As eigenvalues may be roughly approximated in implementation (see Section~\ref{sec:software}), we add a safety factor $\lambda_s := \min\{2\underline{\lambda}_\beta,\lambda_{\min}^0\}-1$ to guarantee $Q - \lambda_sI \succeq 0$.

Even in the case that $\lambda_{\min}^0<\underline{\lambda}_\beta$, we find that replacing $\lambda_{\min}^0$ by $\lambda_s$ still makes $\|x(\hat{u}) - \hat{u}\|$ smaller (see Lemma~\ref{cor:la_s}), which improves the robustness of the fixed-point iterations, as shown in the following Corollary.

\begin{corollary}\label{cor:la_s}
    Consider a nonconvex $\mathrm{MIQP}$ and associated perturbed solution $x^*$ as defined in Proposition~\ref{prop:org_lambda}.
    
    If $\lambda_{\min}^0 <\underline{\lambda}_\beta$, then
    $\|x_\beta^*-\hat{u}_\beta\| \leq \frac{(2\overline{\lambda}_\beta -\lambda_{\min}^0)\|\hat{u}_\beta\| + \overline{\Theta}}{2(\underline{\lambda}_\beta -\lambda_{\min}^0)}$, where $\overline{\Theta}:=\max\{\|2Q_\gamma^T x_\alpha + a_\beta\|\ |\ x_\alpha\in\{0,1\}^{n_\alpha}\}$ as defined in Theorem~\ref{thm:good_lambda}.
\end{corollary}

\begin{proof}
    Replace $-\lambda_s$ with $-\lambda_{\min}^0$ into Equation~(\ref{eq:upper_bound}), then the conclusion is proved. \qed
\end{proof}

Observe that $\lambda_s = -\min\{2\underline{\lambda}_\beta, \lambda_{\min}^0\} - 1 < \lambda_{\min}^0 \leq \underline{\lambda}_\beta \leq \overline{\lambda}_\beta$, $\frac{2\overline{\lambda}_\beta - \lambda_s}{2(\underline{\lambda}_\beta - \lambda_s)} < \frac{2\overline{\lambda}_\beta - \lambda_{\min}^0}{2(\underline{\lambda}_\beta - \lambda_{\min}^0)}$; moreover, $\frac{\overline{\Theta}}{2(\underline{\lambda}_\beta - \lambda_s)} < \frac{\overline{\Theta}}{2(\underline{\lambda}_\beta - \lambda_{\min}^0)}$. Thus $\|x(\hat{u}) - \hat{u}\|$ is smaller if replace $\lambda_{\min}^0$ with $\lambda_s$, which makes the fixed-point iterations more robust.

The relation between $x(\hat{u})$ and $\mbox{Approx}(\hat{u})$ for MIQCP will be too complicated to analyze theoretically. Therefore, we defer the comparison between the original eigenvalue method and our proposed method to the experiments in Section~\ref{sec:comp_matrices}.

\section{Primal Heuristics}\label{sec:primal_heuristics}

In this section, we propose some primal heuristics to detect the initial feasible solutions across different types of MIQCQPs mentioned in Section~\ref {sec:intro}.

\subsection{Mixed-Integer Box-Constrained Quadratic Programs}

For MIBQPs, we adopt Algorithm~\ref{Alg:class1}, where integer fixings are explored by toggling between the ceiling and floor.

\begin{algorithm}[!htbp]
\SetAlgoLined
\LinesNumbered
\SetKwRepeat{Do}{do}{while}
\SetKwInput{Input}{Input}
\SetKwInput{Output}{Output}
\Input{MIBQP $\mathrm{Q}_{\mathrm{MIBQP}}$}
\Output{Feasible Solution $x^*$}
 Solve the integer-relaxed $\mbox{Approx}(u)$ for $\mathrm{Q}_{\mathrm{MIBQP}}$ to get $x^{*}$\;
 \For{$i \in \mathrm{shuffle}(\mathcal{I})$}{
        \If{$x_i^{*}$ is not an integer}{
            Set $x_i^* := \arg\min\{x^TQ^0x+a_0^Tx+b_0\ |\ x_i \in\{\lfloor x_i^{*} \rfloor, \lceil x_i^{*} \rceil\},  x_j = x_j^{*}, \forall j\neq i\}$\;
        }
 }
 \textbf{return} $x^*$
 \caption{Random Flip}
 \label{Alg:class1}
\end{algorithm}

\subsection{Mixed-Integer Quadratic programming}

For MIQP with a quadratic objective function but without any quadratic constraints, we consider two cases: (1) $\mathcal{L}_c=\emptyset$, i.e., there is an equivalent convex quadratic function for the objective, and (2) $\mathcal{L}_c \neq \emptyset$.

\subsubsection{Case 1}
If $\mathcal{L}_c=\emptyset$, $\mbox{Approx}(u)$ is a convex reformulation to the MIQP. We detect the primal feasible solution by Random Flip and then project with Algorithm~\ref{Alg:class2_1}. 

In lines 4 and 6, we temporarily fix $x_i^0$ to its floor or ceiling and run domain propagation to check whether such fixing is feasible or infeasible \cite{fischetti2009feasibility}. If exactly one of the fixings is feasible, we keep such fixing; otherwise, we select whichever one minimizes the objective function in line 9. Finally, in line 18, we project the flipped solution to the feasible region of MIQP. The objective function with $\ell_1$ norm can be linearized (see \cite{fischetti2005feasibility}); thus, this projection is completed by solving an MILP---experiments show most of these can be solved to optimality in seconds.

\begin{algorithm}[!htbp]
\SetAlgoLined
\LinesNumbered
\SetKwRepeat{Do}{do}{while}
\SetKwInput{Input}{Input}
\SetKwInput{Output}{Output}
\Input{MIQP $\mathrm{Q}_{\mathrm{MIQP}}$, feasible region $\mathcal{F}(\mathrm{Q}_{\mathrm{MIQP}})$}
\Output{Feasible Solution $x^*$}
  Solve the integer-relaxed $\mbox{Approx}(u)$ for $\mathrm{Q}_{\mathrm{MIQP}}$ to get $x^{(0)}$\;
 \For{$i \in \mathrm{shuffle}(\mathcal{I})$}{
        \If{$x_i^0$ is not an integer}{
            \uIf{$x_i^0 := \lceil x_i^0 \rceil$ is feasible $\wedge$ $x_i^0 := \lfloor x_i^0 \rfloor$ is infeasible}{
                Set $x_i^0 := \lceil x_i^0 \rceil$\;
            }
            \uElseIf{$x_i^0 := \lceil x_i^0 \rceil$ is infeasible $\wedge$ $x_i^0 := \lfloor x_i^0 \rfloor$ is feasible}{
                Set $x_i^0 := \lfloor x_i^0 \rfloor$\;
            }
            \Else{
                $x_i^0 := \arg\min_{x_i \in\{\lfloor x_i^0 \rfloor, \lceil x_i^0 \rceil\}}\{x^TQ^0x+a_0^Tx+b_0\ |\ x_j = x_j^0, \forall j\neq i\}$\;
            }
        }
 }
 $x^* := \arg\min_x\{\|x - x^{(0)}\|_1\ |\ x\in \mathcal{F}(\mathrm{Q}_{\mathrm{MIQP}})\}$\;
 \textbf{return} $x^*$
 \caption{Random Flip and Project}
 \label{Alg:class2_1}
\end{algorithm}

\subsubsection{Case 2}

If $\mathcal{L}_c\neq \emptyset$, Algorithm~\ref{Alg:class2_1} also returns a feasible solution. However, we leverage the proposed fixed-point iteration method in Section~\ref{sec:eigen_method_miqp} to develop Algorithm~\ref{Alg:class2_2}. 

Line 4 terminates the algorithm if the objective value is not improved. Line 5 updates $\hat{u}$ based on the fixed-point relation from Proposition~\ref{prop:fixed-point}, and we adopt an accelerated fixed-point iteration proposed by Walker and Ni\cite{walker2011anderson}.

\begin{algorithm}[!htbp]
\SetAlgoLined
\LinesNumbered
\SetKwRepeat{Do}{do}{while}
\SetKwInput{Input}{Input}
\SetKwInput{Output}{Output}
\Input{MIQP $\mathrm{Q}_{\mathrm{MIQP}}$, feasible region $\mathcal{F}(\mathrm{Q}_{\mathrm{MIQP}})$}
\Output{Feasible Solution $x^*$}
 Set $\hat{u}:=u$, $obj = \infty$\;
 \For{$i = 1:\mathrm{MaxIter}$}{
    Set $x^*$ as the solution to \texttt{Random Flip and Project} with $\mbox{Approx}(\hat{u})$\;
    \eIf{$obj - ({x^*}^TQ^0x^*+a_0^Tx^*+b_0) > \epsilon$}{
        Set $\hat{u}_j = \alpha\hat{u}_j + 2(1-\alpha)x_j^*$ for all $j\notin\mathcal{I}$\;
        Set $obj = ({x^*}^TQ^0x^*+a_0^Tx^*+b_0)$\;
    }{
        \textbf{Break}\;
    }
}
 \textbf{return} $x^*$
 \caption{Projection and Fixed-Point Iterations}
 \label{Alg:class2_2}
\end{algorithm}

\subsection{Mixed-Integer Quadratically-Constrained Quadratic Programming}
\label{sec:MIQCQPproj}
We propose two different feasibility-pump-like methods for (nonconvex) MIQCQPs and compare them at the end of this section.

\subsubsection{Two-Projection Method} 

We construct Algorithm~\ref{Alg:class3_1} based on two kinds of projections: the projection onto the feasible region of $\mbox{Approx}(\hat{u})$, $\mathcal{F}(\mbox{Approx}(\hat{u}))$, which solves a convex MIQCP and corresponds to the ``projection step'' in \cite{bonami2012heuristics}; and the projection onto the feasible region of the original MIQCQP, $\mathcal{F}(\mathrm{Q}_{MIQCP})$, which solves a (nonconvex) MIQCP and corresponds to the ``rounding step'' in \cite{bonami2009feasibility}. Though both projections are NP-hard, we only aim to get a feasible solution within the time limit (10 seconds) instead of an optimal solution.

Line 1 conducts a local NLP solver for the continuous relaxation of $\mathrm{Q}_{MIQCP}$. Line 2 initializes $\hat{u}$. Line 4 solves a convex MIQCP with a time limit. If line 4 does not find any feasible solution, we increase $\hat{u}$ in line 6 to enlarge the feasible region of $\mbox{Approx}(\hat{u})$. If we get $x^{(1)}$, line 8 projects $x^{(1)}$ onto $\mathcal{F}(\mathrm{Q}_{MIQCP})$, which solves a nonconvex MIQCP with a time limit. If line 8 does not find any feasible solution, we tune $\hat{u}$ based on $x^{(1)}$ and the fixed-point relation. If we get $x^{(2)}$, line 16 fixes all integers in $x^{(2)}$ and locally solves $\mathrm{Q}_{MIQCP}$ to get an improved feasible solution, e.g., $x^*$.

\begin{algorithm}[!htbp]
\SetAlgoLined
\LinesNumbered
\SetKwRepeat{Do}{do}{while}
\SetKwInput{Input}{Input}
\SetKwInput{Output}{Output}
\Input{MIQCQP $\mathrm{Q}_{MIQCP}$, feasible regions $\mathcal{F}(\mathcal{Q})$ and $\mathcal{F}(\mbox{Approx}(\hat{u}))$}
\Output{Feasible Solution $x^*$}
 Locally solve continuous relaxation of $\mathrm{Q}_{MIQCP}$ to get $x^{(0)}$\;
 Set $\hat{u}_i:= 2x^{(0)}_i, \forall i\notin \mathcal{I}$, and $\hat{u}_i:= u_i, \forall i\in \mathcal{I}$\;
 \For{$i=1:\mathrm{MaxIter}$}{
    Set $x^{(1)}$ as the best-detected feasible solution in solving $\arg\min\{\|x - x^{(0)}\|_1:x\in \mathcal{F}(\mbox{Approx}(\hat{u}))\}$\;
    \eIf{Do not get $x^{(1)}$}{
        Set $\hat{u}_i:= 2\hat{u}_i$ for all $i\notin \mathcal{I}$\;
    }{
        Set $x^{(2)}$ as the best-detected feasible solution in solving $\arg\min\{\|x - x^{(1)}\|_1:x\in \mathcal{F}(\mathrm{Q}_{MIQCP})\}$\;
        \eIf{Do not get $x^{(2)}$}{
            Set $\hat{u}_i:=\alpha\hat{u}_i + (1-\alpha)x^{(1)}_i$ for all $i\notin \mathcal{I}$\;
        }{
            \textbf{Break}\;
        }
    }
 }
 Fix all integer variables in $x^{(2)}$ and solve $\mathrm{Q}_{MIQCP}$ with a local solver, obtaining $x^*$\;
 \textbf{return} $x^*$
 \caption{Two-Projection Method}
 \label{Alg:class3_1}
\end{algorithm}

We find that Gurobi \cite{gurobi} and COPT \cite{copt} solvers can detect a feasible solution in line 4 within the time limit of 10 seconds in most instances, while line 8 is challenging. However, after we apply fixed-point iterations \cite{ortega2000iterative,walker2011anderson} for $\hat{u}$, $x^{(1)}$ usually closes to or insides $\mathcal{F}(\mathrm{Q}_{MIQCP})$ after several iterations as we discussed in Section~\ref{sec:eigen_method}. Then we can find a feasible solution in line 8 within the time limit.

\subsubsection{Relaxing Projection Method}
\label{sec:relproj}
Here we develop an alternating heuristic given as Algorithm~\ref{Alg:class3_2}.

In Section~\ref{sec:eigen_method}, we show that if $x(\hat{u})\geq \hat{u}$, we find a feasible solution. Therefore, we propose (FPR1), whose aim is to minimize  $\hat{u}_i -x_i$. Moreover, $x(\hat{u})$ violates constraint~(\ref{eq:eigen_I_qcon}) by a factor of $\lambda^k(\sum_{i\not\in\mathcal{I}}\hat{u}_i\cdot x(\hat{u})_i - x(\hat{u})_i^2)$; thus minimizing $\hat{u}_i -x_i$ can also reduce constraint violation. Additionally, (FPR1) is a convex MIQCQP, and experiments demonstrate it can be solved within the time limit (10 seconds) in most cases.

\begin{equation*}
    \begin{aligned}
        (\mbox{FPR1})\ \min\ &\sum_{i\notin\mathcal{I}} \delta_i\\
        s.t.\quad &x^T(Q^k-\lambda^k I)x + a_k^Tx + \lambda^k\sum_{i\not\in \mathcal{I}}\hat{u}_ix_i + \lambda^k\sum_{i\in \mathcal{I}}X_{ii} \leq b_k, k\in \{1,...,m_1\}\\
        &\delta_i \geq \hat{u}_i -x_i, i\notin \mathcal{I}\\
        &\delta \geq 0\\
        &(\ref{eq:linearcon}),(\ref{eq:boxcon}),(\ref{eq:intcon})
    \end{aligned}
\end{equation*}
If we can find a feasible solution for (FPR1) such that $\sum_{i\notin\mathcal{I}}\delta_i = 0$, then $x_i\geq \hat{u}_i$ for all $i\notin \mathcal{I}$; thus $x$ is a feasible solution for MIQCQP~(\ref{eq:miqcqp}).  Otherwise, we apply (FPR2) to do a further local search. 

(FPR2) contains $m_1$ slack variables for quadratic constraints~(\ref{eq:quadcon}), and the objective function is the $\ell_1$-slack-penalized function. If we can find a solution such that  $\sum_{k\in\{1,...,m_1\}} s_k = 0$, then the solution $x$ is feasible in MIQCQP~(\ref{eq:miqcqp}). Though (FPR2) is a nonconvex MIQCQP, we will fix all integer variables from the solution of (FPR1) and locally solve the NLP.

\begin{equation*}
    \begin{aligned}
        (\mbox{FPR2})\ \min\ &\sum_{k=\{1,...,m_1\}} s_k\\
        s.t.\quad &x^TQ^kx + a_k^Tx  - s_k \leq b_k, k\in \{1,...,m_1\}\\
        &s \geq 0\\
        &(\ref{eq:linearcon}),(\ref{eq:boxcon}),(\ref{eq:intcon})
    \end{aligned}
\end{equation*}

Algorithm~\ref{Alg:class3_2} is based on the two subproblems as follows. Lines 4 and 5 solve (FPR1) and obtain $x^{(1)}$. If (FPR1) is infeasible, we enlarge the feasible region by doubling $\hat{u}$ in line 17. Line 6 checks whether $x^{(1)}$ is feasible for $\mathrm{Q}_{MIQCP}$. Line 9 fixes integer variables in $x^{(1)}$ and locally solves (FPR2), which is a nonconvex QCP, which tries to move $x^{(1)}$ towards the feasible region of MIQCQP~(\ref{eq:miqcqp}) and get $x^{(2)}$. If $x^{(2)}$ is still infeasible, line 13 updates $\hat{u}$ with $x^{(2)}$ and the fixed-point relation.

\begin{algorithm}[!htbp]
\SetAlgoLined
\LinesNumbered
\SetKwRepeat{Do}{do}{while}
\SetKwInput{Input}{Input}
\SetKwInput{Output}{Output}
\Input{MIQCQP $\mathrm{Q}_{MIQCP}$, (FPR1), (FPR2)}
\Output{Feasible Solution $x^*$}
 Locally solve continuous relaxation of $\mathrm{Q}_{MIQCP}$ to get $x^{(0)}$\;
 Set $\hat{u}_i:= 2x^{(0)}_i, \forall i\notin \mathcal{I}$, and $\hat{u}_i:= u_i, \forall i\in \mathcal{I}$\;
 \For{$i=1:\mathrm{MaxIter}$}{
    \eIf{$\mathrm{(FPR1)}$ with $\hat{u}$ is feasible}{
        Solve (FPR1) with $\hat{u}$ obtaining $x^{(1)}$ and $\delta^*$\;
        \eIf{$\sum_{j\notin\mathcal{I}}\delta_j^* = 0$}{
            Set $x^{(3)} := x^{(1)}$, Break\;
        }{
            Fix all integers in $x^{(1)}$ and locally solve (FPR2) obtaining $x^{(2)}$ and $s^*$\; 
            \eIf{$\sum_{k\in\{1,...,m_1\}}s_k^* = 0$}{
                Set $x^{(3)} := x^{(2)}$, Break\;
            }{
                Set $\hat{u}_i:= \alpha\hat{u}_i+(1-\alpha) x^{(2)}_i, \forall i\notin\mathcal{I}$\;
            }
        }
    }{
        $\hat{u}_i:= 2\hat{u}_i, \forall i\notin\mathcal{I}$\;
    }
 }
 Fix all integer variables in $x^{(3)}$ and solve $\mathrm{Q}_{MIQCP}$ with a local solver, obtaining $x^*$\;
 \textbf{return} $x^*$
 \caption{Relaxing Projection Method}
 \label{Alg:class3_2}
\end{algorithm}

Similar to Algorithm~\ref{Alg:class3_1}, Gurobi \cite{gurobi} and COPT \cite{copt} can detect a feasible solution in line 5 within the time limit of 10 seconds in most instances. Furthermore, we can solve (FPR2) in line 9 with a warmstart point $x^{(1)}$, where $s_k = \max\{0, {x^{(1)}}^TQ^kx^{(1)} + a_k^Tx^{(1)} - b_k\}$ has been a feasible solution.

\subsubsection{Comparison of Two-Projection and Relaxing Projection}
Algorithm~\ref{Alg:class3_1} needs to solve both convex and nonconvex MIQCQPs (lines 4 and 8), while Algorithm~\ref{Alg:class3_2} only needs to solve a convex MIQCQP~(FPR1), and locally solve a nonconvex QCP~(FPR2) with fixed integer variables. Therefore, compared to Algorithm~\ref{Alg:class3_1}, Algorithm~\ref{Alg:class3_2} may get solutions faster. 

On the other hand, Algorithm~\ref{Alg:class3_1} projects $x^{(0)}$ to $x^{(1)}$, and projects $x^{(1)}$ to $x^{(2)}$ in every iteration; thus, $x^{(0)}$, the local optima of the continuous relaxation of MIQCQP, is taken into account throughout the entire process. In contrast, Algorithm~\ref{Alg:class3_2} only uses $x^{(0)}$ to find the initial $\hat{u}$, and subsequently does not incorporate any information from the objective function~(\ref{eq:obj}). Therefore, the feasible solution identified by Algorithm~\ref{Alg:class3_1} may be of higher quality than that obtained from Algorithm~\ref{Alg:class3_2}.

As Algorithm~\ref{Alg:class3_1} and \ref{Alg:class3_2} have their own unique advantages, we run them in parallel: if either one finds a feasible solution, we terminate the other one after 1 more iteration.

\subsection{Comparison with Feasibility Pump Literature}\label{sec:comp_methods}

In this subsection, we compare our proposed methods with similar primal heuristics based on the feasibility pump. Algorithm~\ref{Alg:class1} starts from a solution of the continuous relaxation of $\mbox{Approx}(u)$, as with typical feasibility pumps \cite{fischetti2005feasibility,fischetti2009feasibility}; likewise with flipping of binary variables.

Algorithm~\ref{Alg:class2_1} is an extension to Algorithm~\ref{Alg:class1}, in which we want the detection solution close to the start point, which potentially has a better objective value, as well as is integer-feasible. Therefore, when flipping binaries, we conduct a domain propagation to check the feasibility of different assignments of binaries, in the spirit of feasibility pump 2.0 \cite{fischetti2009feasibility}.

Algorithm~\ref{Alg:class2_2} extends Algorithm~\ref{Alg:class2_1} for nonconvex MIQCQP with $\mathcal{L}_c\neq\emptyset$, in which $\mbox{Approx}(u)$ is not an exact convex reformulation. As discussed in Sections~\ref {sec:org_eig_method} and \ref{sec:eigen_method}, we develop a dynamic convex approximation $\mbox{Approx}(\hat{u})$ to refine convex models iteratively, which may result in a better feasible solution with Algorithm~\ref{Alg:class2_1}. To our knowledge, such an approximation scheme for projection is unique in the literature.

Algorithm~\ref{Alg:class3_1} and Algorithm~\ref{Alg:class3_2} stray further from the pumping idea and could be viewed as part of a more general family of alternating heuristics. Algorithm~\ref{Alg:class3_1} may be described as a modified feasibility-pump-style two-projection method, whereas Algorithm~\ref{Alg:class3_2} can be considered a pump-inspired alternating feasibility-repair method. Moreover, both involve our convex approximation updates. In contrast, the MINLP literature on feasibility pumps follow more closely the original MIP template, alternating between continuous relaxation (NLP) feasibility and integrality feasibility via projection, rounding, or mixed-integer distance minimization \cite{d2012storm,belotti2017three}.

\section{Parallel Local Branching}\label{sec:PLB}
To further refine solutions, we adopt the local branching framework of Fischetti and Lodi \cite{fischetti2003local} with two key modifications:  
\begin{enumerate}
    \item we propose a reverse local branching constraint to help the heuristic escape local optimum;
    \item we parallelize the local branching method.
\end{enumerate}

\subsection{Standard and Reverse Local Branching Constraints}  
Fischetti and Lodi \cite{fischetti2003local} proposed the local branching constraint (LBC):  
\begin{equation}
    \Delta(x, \bar{x}) := \sum_{j \in \bar{S}} (1 - x_j) + \sum_{j \in \mathcal{B} \setminus \bar{S}} x_j \leq k,
\end{equation}  
where $\bar{S} := \{j \in \mathcal{B} \mid \bar{x}_j = 1\}$. LBC implies the $k$-nearest-neighborhood of the feasible solution $\bar{x}$. However, this search strategy can sometimes fall into a bad local optimum. Fischetti and Lodi \cite{fischetti2003local} restart local branching with a non-best-known solution.

We propose another way to escape the local optimum by adding a reverse local branching constraint (RLBC):  
\begin{equation}
    \Delta_r(x, \bar{x}) := \sum_{j \in \bar{S}} x_j + \sum_{j \in \mathcal{B} \setminus \bar{S}} (1 - x_j) \leq k.
\end{equation}  
RLBC defines the $k$-farthest-neighborhood of the feasible solution $\bar{x}$, ensuring search in an unexplored region.

\subsection{Parallel Algorithm}  

Suppose there is an MIQCQP, e.g., $\mathrm{Q}_{MIQCP}$ and an initial solution $\bar{x}^0$, which is not necessarily feasible. The parallel local branching is implemented as Figure~\ref{fig:parallel_local_branch}. 

\begin{figure}[!htbp]
    \centering
    \begin{tikzpicture}[scale=.75,auto=left]
    \node[] at (-6.5, 7) {$\Delta(x, \bar{x}^0)\leq k$};
    \node[] at (-0.95, 7) {$\Delta(x, \bar{x}^0)\geq k$};
    \node[] at (-5.5, 3.5) {$\Delta(x, \bar{x}^1)\leq k$};
    \node[] at (1.4, 3.5) {$\Delta(x, \bar{x}^1)\geq k$};
    \node[] at (4.75, 1) {$\Delta_r(x, \bar{x}^1)\leq k$};
    \node[] at (-2, 0.5) {$\Delta(x, \bar{x}^2)\leq k$};
    \node[] at (-6, 4.3) {{\color{red}Improved Solution $\bar{x}^1$}};
    \node[] at (-4, 1.3) {{\color{red}No Improved Solution}};
    \node[] at (4, -1.8) {{\color{red}Improved Solution $\bar{x}^2$}};
    
        \node[circle,draw] (r0) at (-3,8) {$R_0$};
        \node[rectangle,draw] (a1) at (-8.3,5) {1.1};
        \node[rectangle,draw] (a2) at (-6.8,5) {1.2};
        \node[rectangle,draw] (a3) at (-5.3,5) {1.3};
        \node[rectangle,draw] (a4) at (-3.7,5) {1.4};
        \node[circle,draw] (r1) at (-1,5) {$R_1$};
        \node[rectangle,draw] (b1) at (-6.3,2) {2.1};
        \node[rectangle,draw] (b2) at (-4.8,2) {2.2};
        \node[rectangle,draw] (b3) at (-3.3,2) {2.3};
        \node[rectangle,draw] (b4) at (-1.8,2) {2.4};
        \node[circle,draw] (r2) at (1,2) {$\bar{R}_1$};
        \node[rectangle,draw] (c1) at (6.3,-1) {3.1};
        \node[rectangle,draw] (c2) at (4.8,-1) {3.2};
        \node[rectangle,draw] (c3) at (3.3,-1) {3.3};
        \node[rectangle,draw] (c4) at (1.8,-1) {3.4};
        \node[circle,draw] (r3) at (-2,-1) {$\bar{R}_2$};
        \node[] at (-2,-2) {$\cdots$};
        \draw[-,black] (r0) -- (-6,6.25);
        \draw[-,black] (-6, 6.25) -- (a1);
        \draw[-,black] (-6, 6.25) -- (a2);
        \draw[-,black] (-6, 6.25) -- (a3);
        \draw[-,black] (-6, 6.25) -- (a4);
        \draw[-,black] (r0) -- (r1);
        \draw[-,black] (r1) -- (-4,3.25);
        \draw[-,black] (-4, 3.25) -- (b1);
        \draw[-,black] (-4, 3.25) -- (b2);
        \draw[-,black] (-4, 3.25) -- (b3);
        \draw[-,black] (-4, 3.25) -- (b4);
        \draw[-,black] (r1) -- (r2);
        \draw[-,black] (r2) -- (4,0.25);
        \draw[-,black] (4,0.25) -- (c1);
        \draw[-,black] (4,0.25) -- (c2);
        \draw[-,black] (4,0.25) -- (c3);
        \draw[-,black] (4,0.25) -- (c4);
        \draw[-,black] (r2) -- (r3);

        \draw[dashed, draw = red, rounded corners=10pt] (-9, 6) rectangle (-3, 4);
        \draw[dashed, draw = red, rounded corners=10pt] (-7, 3) rectangle (-1, 1);
        \draw[dashed, draw = red, rounded corners=10pt] (1, 0) rectangle (7, -2.2);

        \draw[draw = black] (1.1, 8.5) rectangle (7.1, 5.2);
        \draw[dashed, draw = red, rounded corners=4pt] (1.35, 8.25) rectangle (3, 7.5);
        \node[] at (5, 7.85) {Parallel Computing};
        \draw[draw = black] (1.35, 7.25) rectangle (3, 6.5);
        \node[] at (5, 6.85) {Subproblem (LBCs)};
        \node[circle, draw] at (2, 5.85) {$R_i$};
        \node[] at (5, 5.85) {Reduced Problem};
    \end{tikzpicture}
    \caption{Caption}
    \label{fig:parallel_local_branch}
\end{figure}

In the root node $R_0$, we have $\mathrm{Q}_{MIQCP}$ and $\bar{x}^0$.

In the left branch, we add an LBC, $\Delta(x, \bar{x}^0)\leq k$, to $\mathrm{Q}_{MIQCP}$, then get a subproblem. The classical local branching sets $k$ around $18$ and solves the subproblem \cite{fischetti2003local}. In our parallel local branching, we split the subproblem into 4 parts solved in parallel, each incorporating two LBCs:
\[
\Delta(x, \bar{x}) \geq k_1^{(i)}, \quad \Delta(x, \bar{x}) \leq k_2^{(i)}, i=1,2,3,4.
\]   
The values of $(k_1^{(i)}, k_2^{(i)})$ for subproblems $i = 1,2,3,4$ are set as $(1,7)$, $(8,13)$, $(14,17)$, and $(18,19)$, respectively. This partitioning is designed to balance the computational workload, because the regions farther from $\bar{x}$ tend to have a larger search space.

We solve the four subproblems in parallel with a time limit, and get the best solution $\bar{x}^1$. If $\bar{x}^1$ is better than $x^0$, we go to node $R_1$.

In node $R_1$, we have a reduced problem $\mathrm{Q}_{MIQCP}$ with an additional LBC,
$\Delta(x, \bar{x}^0)\geq k$ and an improved solution $\bar{x}^1$.

Then, in the left branch of node $R_1$, we repeat the step from the left branch of the root node $R_0$. If we cannot find an improved solution, we will check whether all subproblems have been solved or terminated due to the time limit. If one subproblem stops due to the time limit, we can split the subproblem with LBCs and solve them again. 

If all subproblems have been solved but with no improved solution, we will adopt RLBCs, e.g., $\Delta_r(x,\bar{x}^1)\leq k$ to search the farthest neighborhood of $\bar{x}^1$.

\section{Algorithm and Experiments}\label{sec:experiment}

This section describes our implementation and experiments.  All code is publicly available at \href{https://github.com/foreverdyz/primalheuristic\_miqcqp}{repository}\footnote{https://github.com/foreverdyz/primalheuristic\_miqcqp}. Our algorithm is implemented as shown in Chart~\ref{fig:algorithm}.

\begin{Chart}[!htbp]
\centering
    \begin{tikzpicture}[scale=.75, node distance = 1.5cm]
        \node (Input) [startstop] {Input MIQCQP};
        \node (classify) [process, below of=Input] {Categorize the problem, Section~\ref{sec:intro}};
        \node (NL) [process, below of=classify, xshift=-3cm] {MIBQP};
        \node (WL) [process, below of=classify] {MIQP};
        \node (QCP) [process, below of=classify, xshift= 3cm] {MIQCP};
        \node (NLH) [process, below of=NL] {Alg.~\ref{Alg:class1}};
        \node (WLH) [process, below of=WL] {Alg.~\ref{Alg:class2_1}/\ref{Alg:class2_2}};
        \node (QCPH) [process, below of=QCP] {Alg.~\ref{Alg:class3_1} and \ref{Alg:class3_2} in parallel};
        \node (output) [process, below of=WLH] {Parallel local branching, Section~\ref{sec:PLB}};
        \node (Output) [startstop, below of = output] {Output the best-found solution};
        \draw [arrow] (Input) -- (classify);
        \draw [arrow] (classify) -- (NL);
        \draw [arrow] (classify) -- (WL);
        \draw [arrow] (classify) -- (QCP);
        \draw [arrow] (NL) -- (NLH);
        \draw [arrow] (WL) -- (WLH);
        \draw [arrow] (QCP) -- (QCPH);
        \draw [arrow] (NLH) -- (output);
        \draw [arrow] (WLH) -- (output);
        \draw [arrow] (QCPH) -- (output);
        \draw [arrow] (output) -- (Output);
    \end{tikzpicture}
    \caption{Proposed Primal Heuristics}
    \label{fig:algorithm}
\end{Chart}

We first classify the input problem into one of the three categories mentioned in Section~\ref{sec:intro}, and detect the first feasible solution with the corresponding primal heuristics proposed in Section~\ref{sec:primal_heuristics}. After getting a feasible solution, parallel local branching from Section~\ref{sec:PLB} will further improve the solution.

\subsection{Experimental Setup}

\subsubsection{Test Set}
Experiments are conducted on the set of the QPLIB Collection \cite{furini2019qplib}, consisting of $453$ instances. We select 319 discrete instances and split them into three groups:

\textbf{MIBQP}, 23 instances with neither quadratic nor general linear constraints;

\textbf{MIQP}, 121 instances with linear constraints but not quadratic constraints;

\textbf{MIQCQP}, 175 instances with quadratic constraints.

\subsubsection{Software and Hardware}\label{sec:software}
All algorithms are implemented in Julia 1.11.6 \cite{bezanson2017julia} and a desktop running 64-bit Windows 11 with an AMD Ryzen Threadripper PRO 5975WX. We use 32 threads (1 thread per core) and a memory limit of 32 GB. 

We solve all MIQCQPs, MILPs, and LPs with Gurobi 12.0.1 \cite{gurobi}, some QCQPs with Ipopt 3.14.17 \cite{ipopt}, and with the modeling language JuMP 1.23.2 \cite{jump}. All solvers are deployed with the default setting.

We use \texttt{Arpack.jl} based on ARPACK library \cite{lehoucq1998arpack} to approximate the minimum and maximum eigenvalues of matrices.

\subsubsection{Parameter Setting}\label{sec:parameter}
For Algorithm~\ref{Alg:class2_2}, Algorithm~\ref{Alg:class3_1}, Algorithm~\ref{Alg:class3_2}, we set set the accelerated fixed-point iteration \cite{walker2011anderson} parameter $\alpha = 0.5$. Since we run Algorithm~\ref{Alg:class3_1} and Algorithm~\ref{Alg:class3_2} in parallel, we use 8 threads when solving all projections or other subproblems with Gurobi. In parallel local branching, we explore 4 subproblems in parallel in each branch and use Gurobi for each subproblem with 8 threads and 8 GB memory. We set a global runtime limit of $300$ seconds as the requirement of \href{https://github.com/foreverdyz/primalheuristic\_miqcqp}{Land-Doig MIP Computational Competition 2025}\footnote{https://www.mixedinteger.org/2025/competition}.

\subsubsection{Measurements}
We collect the following measurements to compare our methods with different setups and Gurobi.

\textbf{Found}. The number of instances for which at least one feasible solution is found.

\textbf{Gap\%}. Let $v^*$ be the best-known objective value, and $v$ be the best-detected objective from our method or Gurobi. The primal gap is calculated as $$\frac{|v-v^*|}{\max\{|v|,|v^*|\}}\times 100.$$

\textbf{$\epsilon$-Gap}. The number of instances for which the final gap is less than $1e-4$.

\textbf{PI}. Primal integral \cite{berthold2015heuristic} evaluates the primal gap over time. Let $t_0 = 0$ be the start of the search process, $t_1,...,t_{s-1}$ be times when a new solution is found with gaps $g_1, ..., g_{s-1}$, and $t_s = T$ be the end of the algorithm. Then the primal integral is defined as $$\sum_{i=2}^sg_{i-1}\cdot(t_i-t_{i-1}) + t_1-t_0.$$

\textbf{FFT}. The time to find the first feasible solution.

\textbf{Same, Better, Worse}. Consider two methods obtaining objective values $v_1, v_2$ respectively on the same instance. If $\frac{|v_1-v_2|}{\max\{|v_1|,|v_2|\}}\leq 1e-5$, we count two methods as ``same'' for this instance. Otherwise, when the sense of the problem is min, if $v_1 < v_2$, the first method is better; if $v_1 > v_2$, the first method is worse.

\textbf{Geometric Mean}. All averaged values in the following tables are calculated by the geometric mean with a shift of 1. 

\subsection{Solution Summary}
In Table~\ref{tab:res_sum}, we summarize the performance of the proposed method on instances from QPLIB in three categories. For each category, we report the number of instances for that we find at least a feasible solution, the average primal gap, the number of instances with $\epsilon$-Gap, the average primal integral, and the average time to find the first feasible solution.

\begin{table}[!htbp]
	\centering
	\setlength{\tabcolsep}{2mm}
	\begin{threeparttable}
		\begin{tabular}{lccccc}
            \toprule[1pt]
            Category &Found  &Gap(\%) &$\epsilon$-Gap\footnote{1} &PI &FFT \\
            \midrule
            MIBQP &23/23 &1.43 &9/23 &3.43 &1.18\\
            MIQP &121/121 &1.55 &93/121 &7.71 &3.12\\
            MIQCQP &146/175 &3.05 &95/175 &21.60 &12.98\\
            \midrule
            Total &290/319 &2.24 &197/319 &12.79 &6.310\\
            \bottomrule[1pt]
        \end{tabular}
		\begin{tablenotes}
			\footnotesize
			\item[1] We set a tolerance of $\epsilon$-Gap as $1e-4$.
		\end{tablenotes}
	\end{threeparttable}
	\caption{Performance of the heuristic on QPLIB instances.}
	\label{tab:res_sum}
\end{table}

Our method finds at least one feasible solution for all MIQP instances and for most MIQCQPs. The time to find the first feasible solution is around 5 seconds on average (there is a shift of 1 in values from the table). The effectiveness of the parallel local branching is shown via the primal gap and $\epsilon$-Gap.

\subsection{Comparison between Perturbing Matrix Selections}\label{sec:comp_matrices}

We now assess the impact of the selection of perturbing matrices, which were analyzed theoretically in Section~\ref{sec:matrix_select}. We run the proposed method with both original eigenvalue method, e.g., $P^k = -\min\{0,\lambda_{\min}^{k}\}I$ and our proposed perturbing matrix, $P^{k}:=-\min\{0,\lambda_{s}^{k}\}I$ on MIQCQPs, where $\lambda_{s}^{k} := \min\{2\underline{\lambda}_{\beta}^k, \lambda_{\min}^{k}\} - 1$, where $\underline{\lambda}_{\beta}^k$ is defined as Theorem~\ref{thm:good_lambda} for $Q^k$. Note that the selection of the perturbing matrix only impacts  Algorithm~\ref{Alg:class3_1} and Algorithm~\ref{Alg:class3_2}---these are used to find initial feasible solutions with subsequent solutions found via parallel local branching. 

Table~\ref{tab:comp_matrix} reports all aforementioned indexes for the two types of perturbing matrices. The modified perturbing matrix can find feasible solutions for more instances with a shorter time to find the first feasible solution, which corresponds to our theoretical analysis. After parallel local branching, the reported Gap(\%) and PI from the modified perturbing matrices are still better than the original eigenvalue method, while the original eigenvalue method performs better in $\epsilon$-Gap.

\begin{table}[!htbp]
\centering
\setlength{\tabcolsep}{1.5mm}
\caption{Performance Comparison between Modified Eigenvalue Method and Original Eigenvalue Method for MIQCQPs}
\begin{threeparttable}
\begin{tabular}{lccccc}
\toprule
\textbf{Category} &Found &Gap(\%) &$\epsilon$-Gap\footnote{3} &PI &FFT \\
\midrule
Modified   &146 &3.05 &95 &21.60 &12.98\\
Original   &139 &3.33 &99 &24.13 &14.87\\
\bottomrule
\end{tabular}
\begin{tablenotes}
    \footnotesize
    \item[3] We set a tolerance of $\epsilon$-Gap as $1e-4$.
\end{tablenotes}
\end{threeparttable}
\label{tab:comp_matrix}
\end{table}

For detailed comparison of two types of perturbing matrices, we present the distribution of the time to first feasible solution and the primal integral in Figure~\ref{fig:mod_comp}. Our modified method can complete more instances with a given FFT or PI than the original eigenvalue method.

\begin{figure}[!htbp]
    \centering
    \includegraphics[width=0.49\linewidth]{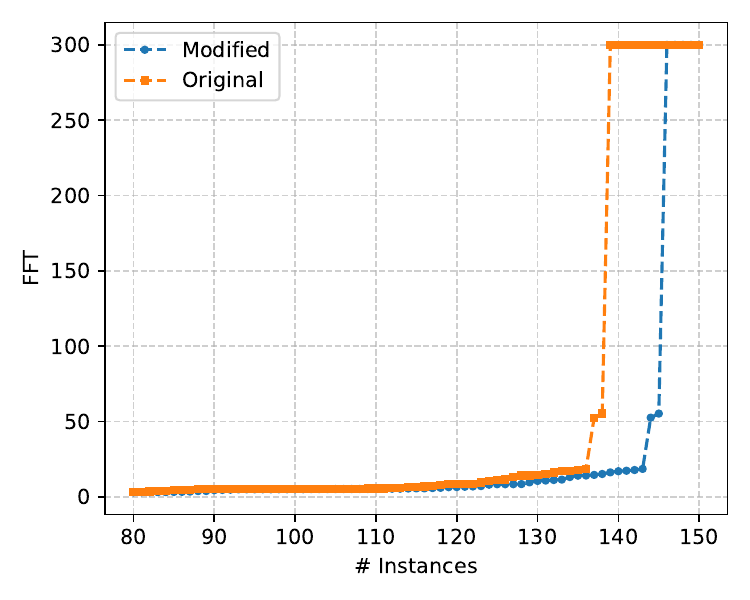}
    \includegraphics[width=0.49\linewidth]{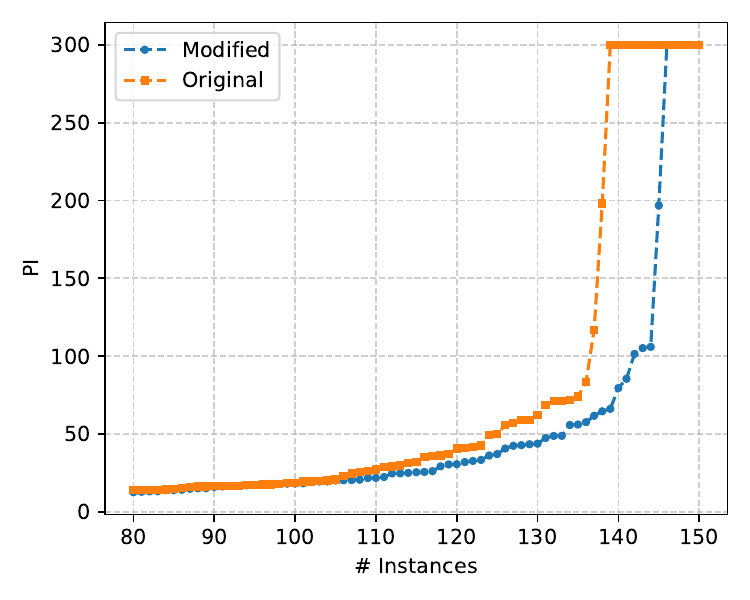}
    \caption{Distribution of the time to first feasible solution, and primal integral over modified and original eigenvalue methods on MIQCQPs.}
    \label{fig:mod_comp}
\end{figure}

\subsection{Impact of Parallelization}

In this section, we compare the parallel local branching with different numbers of parallel subproblems in each branch. We note that algorithms compared in this section use the same proposed primal heuristics to detect the first feasible solution. Therefore, we ignore the terms ``Found'' and ``FFT'' used in previous experiments.

Table~\ref{tab:parallel_comp} summarizes the performances of parallel local branching with 1, 2, 3, and 4 parallel subproblems, and 1-parallel is the original, no parallel, local branching. It implies more numbers of parallel subproblems result in a better final solution (less primal gap) and a better primal integral.

\begin{table}[!htbp]
\centering
\setlength{\tabcolsep}{1.5mm}
\caption{Performance Comparison between Different Numbers of Threads}
\begin{threeparttable}
\begin{tabular}{lccccccccc}
\toprule
& \multicolumn{2}{c}{\textbf{MIBQP}} & \multicolumn{2}{c}{\textbf{MIQP}}& \multicolumn{2}{c}{\textbf{MIQCQP}} & \multicolumn{3}{c}{\textbf{Total}}\\
\cmidrule(lr){2-3} \cmidrule(lr){4-5} \cmidrule(lr){6-7} \cmidrule(lr){8-10}
\textbf{Category} &Gap(\%) &PI &Gap(\%) &PI &Gap(\%) &PI &$\epsilon$-Gap\footnote{1} &Gap(\%) &PI \\
\midrule
4-Parallel   &1.43 &3.43 &1.55 &7.71 &3.05 &21.60 &193 &2.24 &12.79\\
3-Parallel   &1.49 &3.46 &1.59 &7.43 &3.26 &22.98 &183 &2.34 &13.04 \\
2-Parallel   &2.97 &4.63 &1.97 &7.62 &3.86 &23.27 &146 &3.05 &14.22 \\ 
1-Parallel   &6.319 &17.05 &4.54 &13.03 &7.23 &32.74&96 &6.04 &22.01 \\
\bottomrule
\end{tabular}
\begin{tablenotes}
    \footnotesize
    \item[2] We set a tolerance of $\epsilon$-Gap as $1e-4$.
\end{tablenotes}
\end{threeparttable}
\label{tab:parallel_comp}
\end{table}

For detailed comparison, Figure~\ref{fig:parallel_comp} shows the distribution of the optimality gap: the number of instances solved under a given optimality gap, and the distribution of the primal integral over MIQCQPs for all settings. This shows that parallel local branching can improve upon original local branching, but the rate of improvement diminishes as the number of parallel subproblems increases.

\begin{figure}[!htbp]
    \centering
    \includegraphics[width=0.49\linewidth]{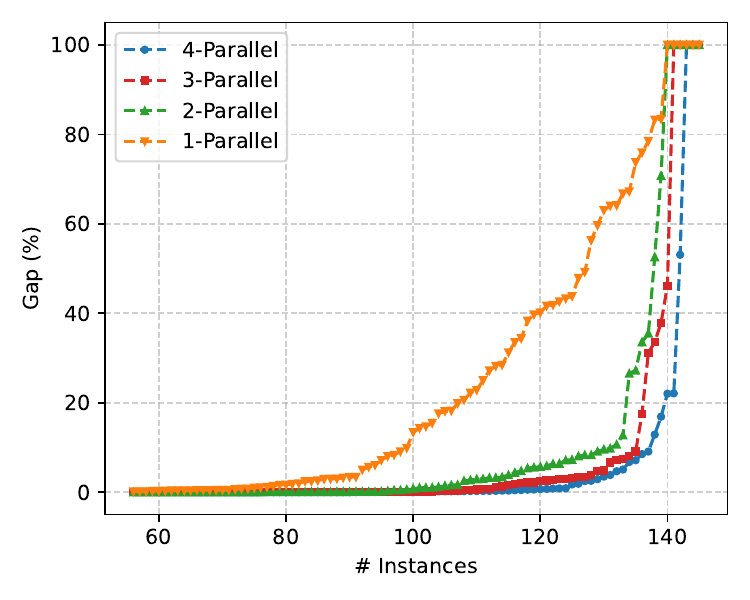}
    \includegraphics[width=0.49\linewidth]{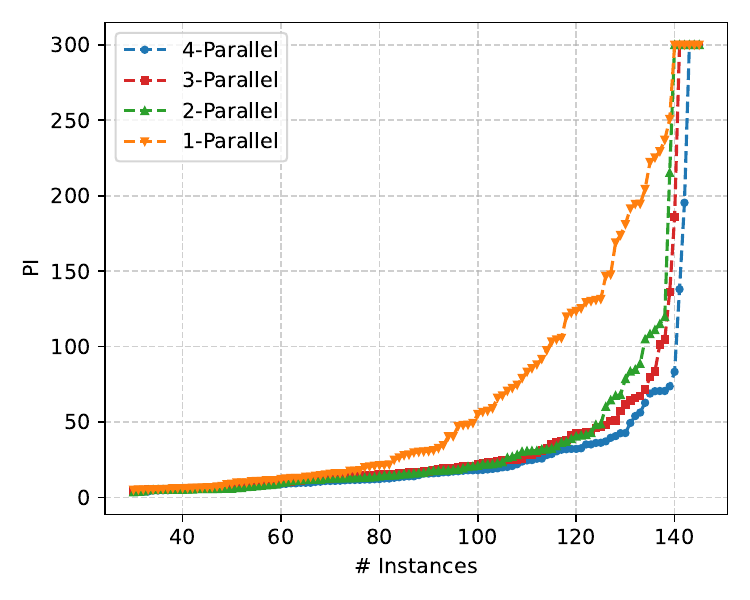}
    \caption{Distribution of the optimality gap, and primal integral over MIQCQP with different parallel settings on MIQCQPs.}
    \label{fig:parallel_comp}
\end{figure}

\subsection{Compare with Gurobi}
To estimate our heuristic relative to state-of-the-art solvers, we compare its performance against Gurobi 12.0.1 on the QPLIB benchmark instances. Gurobi is run with its default setting and a time limit of 300 seconds. Table~\ref{tab:comp_gurobi_point} summarizes the point-to-point comparison between our method and Gurobi with the first detected solution and the final solution, and Table~\ref{tab:comp_gurobi_general} summarizes the overall performance comparison between our method and Gurobi across the three categories. Our method can find better initial solutions in more instances than Gurobi, but takes relatively more time to find them as shown in FFT in Table~\ref{tab:comp_gurobi_general}. For the final solution, Gurobi performs better in point-to-point comparison, while Gurobi also achieves a better primal gap, as Gap(\%) in Table~\ref{tab:comp_gurobi_general}. Thus Gurobi with default setting performs better than our method in some overall sense, although Table.~\ref{tab:comp_gurobi_point} demonstrates that Gurobi and our method are advantageous in different instances. For example, for MIQCQPs, Gurobi finds feasible solutions for more instances than us, e.g. $151$ vs. $146$. However, there are still 7 instances in which our method finds better solutions. Therefore, our method could be a potentially complementary.

\begin{table}[!htbp]
\centering
\setlength{\tabcolsep}{1.5mm}
\caption{Point-to-point comparison of performance vs. Gurobi}
\begin{tabular}{lcccccccc}
\toprule
& \multicolumn{4}{c}{\textbf{First Solution}} & \multicolumn{4}{c}{\textbf{Final Solution}} \\
\cmidrule(lr){2-5} \cmidrule(lr){6-9}
\textbf{Category} & Same & Better & Worse & None & Same & Better & Worse & None \\
\midrule
MIBQP   &1 &20 &2 & 0   &10 &1 &12 & 0  \\
MIQP   &20 &69 &32 & 0  &89 &8 &24 & 0  \\
MIQCQP    &0 &87 &71 &17  &87 &26 &45 &17\\
\bottomrule
\end{tabular}
\label{tab:comp_gurobi_point}
\end{table}

\begin{table}[!htbp]
\centering
\setlength{\tabcolsep}{1.2mm}
\caption{Overall comparison of performance vs. Gurobi}
\begin{tabular}{lcccccccccc}
\toprule
 & \multicolumn{5}{c}{\textbf{Proposed Method}} & \multicolumn{5}{c}{\textbf{Gurobi}} \\
\cmidrule(lr){2-6} \cmidrule(lr){7-11}
\textbf{Category} &Found &Gap(\%) &$\epsilon$-Gap &PI & FFT &Found &Gap(\%) &$\epsilon$-Gap &PI & FFT \\
\midrule
MIBQP  &23 &1.43 &9 &3.43 &1.18 &23 &1.00 &16 &4.38 &1.02 \\
MIQP  &121 &1.55 &93 &7.71 &3.12 &121 &1.01 &100 &3.91 &1.65 \\
MIQCQP   &146 &3.05 &95 &21.60 &12.98 &151 &1.12 &96 &13.9 &6.91 \\
\midrule
Total    &290 &2.24 &197 &12.79 &6.31 &295 &1.07 &212 &7.91 &3.48\\
\bottomrule
\end{tabular}
\label{tab:comp_gurobi_general}
\end{table}

Figure~\ref{fig:gurobi_comp} shows the distribution of time to first feasible solution, i.e., given the FFT, how many instances the method can find a feasible solution, and the primal integral, i.e., how many instances achieve a lower PI than the given one for the method, over MIQCQPs for both our proposed method and Gurobi. We find that Gurobi finds the first feasible solution earlier or achieves a lower primal integral in most cases. However, when given a larger FFT or PI, our methods can solve a bit more instances under the given FFT or PI.  

\begin{figure}[!htbp]
    \centering
    \includegraphics[width=0.49\linewidth]{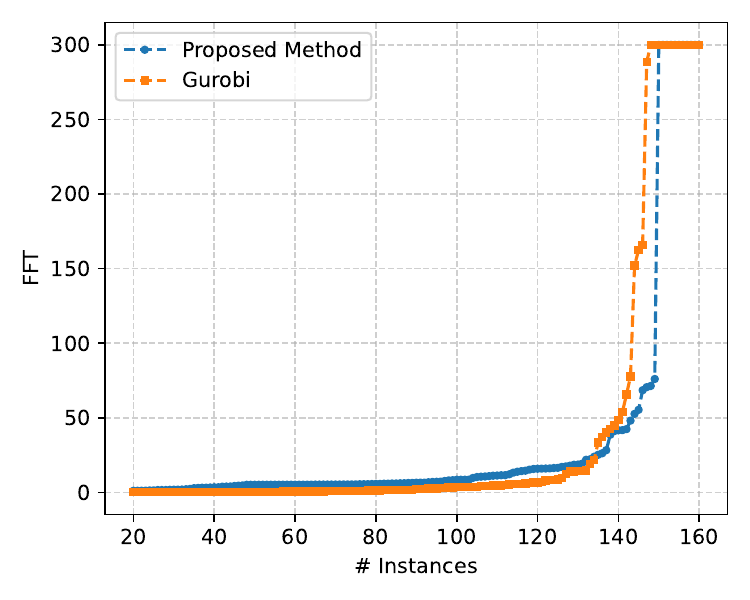}
    \includegraphics[width=0.49\linewidth]{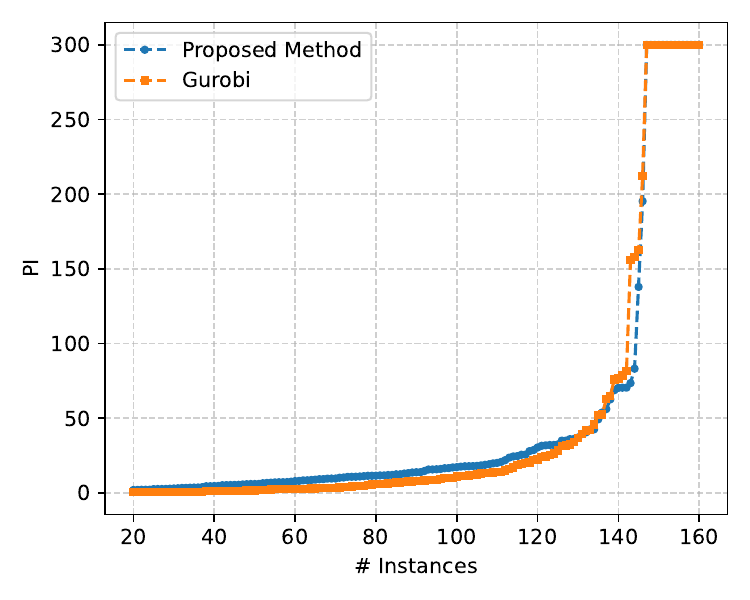}
    \caption{Distribution of the time to first feasible solution, and primal integral over Gurobi and our proposed method on MIQCQPs.}
    \label{fig:gurobi_comp}
\end{figure}

\subsection{Competition and Improvements on QPLIB}

Our heuristic achieved second place in \href{https://github.com/foreverdyz/primalheuristic\_miqcqp}{Land-Doig MIP Computational Competition 2025}\footnote{https://www.mixedinteger.org/2025/competition}, demonstrating its effectiveness on challenging MIQCQPs. Furthermore, the heuristic found 18 new best-known solutions for QPLIB instances, which are listed in Table~\ref{tab:tab1} and can be found at \href{https://qplib.zib.de/}{QBLIB}\footnote{https://qplib.zib.de/}; 3 of these are for instances with no reported feasible solutions. We highlight that this was achieved within the five-minute time limit of the competition.

\begin{table}[!htbp]
\centering
\caption{\label{tab:tab1}Improved solutions for QPLIB instances}
\setlength{\tabcolsep}{1.5mm}{
\begin{tabular}{lcccc}
\toprule[1pt]
Instance &Obj. Sense  &Previous Best &New Solution &Improvement \\

\hline
QPLIB\_0678 &min &- &48000000.0 &-\\
QPLIB\_0689 &max &112.42 &115.0613 &2.35\%\\
QPLIB\_0696 &min &1091160.649 &1086187.1370 &0.46\%\\
QPLIB\_2206 &min &15 &13.0 &15.38\%\\
QPLIB\_3522 &min &84.9179 &84.6250 &0.35\%\\
QPLIB\_3525 &min &4028.3399	&4026.5837 &0.04\%\\
QPLIB\_3582 &min &469654.8988 &469619.8376 &0.007\%\\
QPLIB\_3631 &min &491.3797	&486.0999 &1.09\%\\
QPLIB\_3662 &min &562656.5828 &562617.8818 &0.006\%\\
QPLIB\_3699 &max &-	&3128.5581 &-\\
QPLIB\_3713 &min &85.8838 &85.0071 &1.03\%\\
QPLIB\_3726 &max &-	&3483.5829 &-\\
QPLIB\_3798 &min &111.2686 &110.3276 &0.85\%\\
QPLIB\_3809 &min &472128.5315 &472093.0780 &0.007\%\\
QPLIB\_5023 &min &495034.3674 &495029.0010 &0.001\%\\
QPLIB\_5442 &min &3772309156.0 &1346455159.0 &180.2\%\\
QPLIB\_5925 &min &124.2855 &124.0913 &0.16\%\\
QPLIB\_10022 &min &1348095.635	&1070374.1610 &25.95\%\\
\bottomrule[1pt]
\end{tabular}}
\end{table}

\bibliographystyle{splncs04}
\bibliography{references}

\end{document}